\documentclass[12pt]{amsart}

\usepackage{amssymb, amsmath}
\usepackage{amsthm}
\usepackage{graphicx}
\usepackage{hyperref}
\hypersetup{
    colorlinks=true,
    linkcolor=blue,
    urlcolor=blue,
    }
\usepackage{comment}
\usepackage{amsrefs}

\usepackage{caption}
\usepackage{subcaption}
\usepackage{xcolor}

\theoremstyle{plain} 
\newtheorem{thm}{Theorem}[section]
\newtheorem{lem}{Lemma}[section]
\newtheorem{prop}{Proposition}[section]
\newtheorem{cor}{Corollary}[section]

\theoremstyle{definition}
\newtheorem{dfn}[thm]{Definition}

\theoremstyle{remark}
\newtheorem{rmk}[thm]{Remark}

\newcommand{\CP}{\mathbf C P^2}
\newcommand{\Z}{\mathbf Z}

\newcommand{\Q}{\mathbf Q}
\newcommand{\spn}[1]{\langle #1 \rangle}
\newcommand{\pvec}[2]{\begin{pmatrix}#1\\#2\end{pmatrix}}
\DeclareMathOperator{\Mod}{Mod}
\DeclareMathOperator{\GL}{GL}
\DeclareMathOperator{\SL}{SL}

\begin{document}
\title[Genus-3 Lefschetz pencils on torus bundles]{Genus-3 Lefschetz pencils on torus bundles}

\author[E. Huang]{Evan Huang}
\address{Department of Mathematics\\
Princeton University\\
 Princeton, NJ, 08544}
\email{evan.huang@princeton.edu}

\begin{abstract}
We use the breeding technique and the reverse engineering technique to construct genus-3 Lefschetz pencils for which the total spaces have $c_1=0$. We compute the fundamental groups of our Lefschetz pencils and identify them with the fundamental groups of certain torus bundles over tori. Using work of Friedl and Vidussi, the total spaces of our pencils are homeomorphic to torus bundles over tori. These bundles do not admit sections. As a consequence of the reverse engineering process, we obtain an infinite family of genus-2 pencils, each equivalent to the genus-2 pencils of type $(4,3)$ of Baykur-Korkmaz, Hamada, and Xiao. 
\end{abstract}

\maketitle
\section*{Introduction}
The notion of Kodaira dimension can be extended from K\"ahler surfaces to closed symplectic 4-manifolds \cite{li}, and takes values in $\{-\infty,0,1,2\}$. While the definition of Kodaira dimension $\kappa$ depends on a the symplectic form, it turns out to be a smooth invariant. Manifolds with $\kappa=-\infty$ are completely classified: they are precisely the rational and ruled surfaces (and their blowups) \cite{li2}.

Symplectic 4-manifolds with $\kappa=0$ are called \textit{symplectic Calabi-Yau surfaces} (SCYs), and they are characterized by having a minimal model with torsion canonical class. Minimal SCYs are known to have the rational homology type of the K3 surface, Enriques surface, or a torus bundle over a torus \cite{li}, \cite{bauer}. There is a folklore conjecture that every minimal SCY with $b_1>0$ is diffeomorphic to some torus bundle over a torus.

One way of constructing SCYs is to find a Lefschetz pencil with the right number of base points or equivalently, a Lefschetz fibration with the right number of disjoint $(-1)$-sections \cite{bh}. Baykur \cite{baykur} constructed a family of SCYs using genus-3 pencils and showed that his examples cover all possible integral homology types of torus bundles over tori. In addition, Baykur showed that one of his SCYs was homeomorphic to $T^4$.

Going the other way, Smith showed that if $X$ is the total space of a torus bundle over a torus that has a section, then $X$ admits a genus-3 Lefschetz pencil with 4 base points \cite{smith}. In \cite{hh}, Hamada-Hayano worked out the vanishing cycles of Smith's construction in the case $X=T^4$ and showed this pencil to be equivalent to Baykur's. They also constructed, for each torus bundle $X$ with a section, a genus-3 pencil on a SCY homeomorphic to $X$.

In this paper, we use genus-3 pencils to construct SCYs that are homeomorphic to torus bundles over tori without sections.

\section*{Acknowledgements}
The author would like to thank Zolt\'an Szab\'o for his encouragement and helpful conversations. The author would also like to thank {\.I}nan{\c{c}} Baykur for helpful conversations.

\section*{AI Statement}
The substitution $t=b_2^{-2}a_2^{-1}a_3$ in Section \ref{sec:reverse} was found by GPT 6.2-Sol. The author takes full responsibility for the mathematical content and any errors present.

\section{Torus bundles over tori}
We start by recalling some facts about torus bundles from \cite{sf}. Given $A,B\in\GL(2,\Z)$ with $AB=BA$ and integers $m,n$, there is a $T^2$-bundle over $T^2$ denoted $M(A,B;m,n)$, where $(A,B)$ is the monodromy of the bundle and $(m,n)$ represents the obstruction to having a section. Every $T^2$-bundle over $T^2$ is isomorphic to some $M(A,B;m,n)$, and in addition, we can assume $B=\pm I$. More precisely, every $T^2$-bundle over $T^2$ is isomorphic to some $M(A,I;m,n)$ or $M(A,-I;m,n)$.

\begin{prop}
\label{prop:fundgrp}
(\cite{sf}, Section 1) Let $A=\begin{pmatrix} a & b\\c & d\end{pmatrix},B=\begin{pmatrix} p & q\\r & s\end{pmatrix}\in\GL(2,\Z)$ with $AB=BA$ and $m,n\in\Z$. Put $M=M(A,B;m,n)$. The homotopy exact sequence of a fiber bundle gives the following short exact sequence of groups:
\begin{equation}
\label{eq:ses}
1\to\pi_1F\to\pi_1M\to\pi_1S\to 1.
\end{equation}
$\pi_1M$ is generated by $\alpha,\beta,\sigma,\tau$ with the relations
\begin{align*}
	&\sigma\tau=\tau\sigma & &\alpha^{-1}\sigma\alpha=\sigma^a\tau^c & &\beta^{-1}\sigma\beta=\sigma^p\tau^r\\
	&\alpha\beta\alpha^{-1}\beta^{-1}=\sigma^m\tau^n & &\alpha^{-1}\tau\alpha=\sigma^b\tau^d & &\beta^{-1}\tau\beta=\sigma^q\tau^s.
\end{align*}
Moreover, $\sigma,\tau$ come from the inclusion $\pi_1F\to\pi_1M$ and generate $\pi_1F$, and under the projection $\pi_1M\to\pi_1S$, $\alpha,\beta$ map to generators.
\end{prop}

\begin{prop}
\label{prop:section}
Let $M=M(A,B;m,n)$, with $A=\begin{pmatrix} a & b\\c & d\end{pmatrix},B=\begin{pmatrix} p & q\\r & s\end{pmatrix}\in\SL(2,\Z)$. Then $M$ has a section if and only if $\binom{m}{n}\in\spn{A-I,B-I}$, the span of the columns of $A-I$ and $B-I$.
\end{prop}
\begin{proof}
We will use the presentation for $\pi_1M$ given in Proposition \ref{prop:fundgrp}. From the relations between $\alpha$ and $\sigma$, we see that $\alpha^{-1}\sigma^d\alpha=\sigma^{ad}\tau^{cd}$ and $\alpha^{-1}\tau^{-c}\alpha=\sigma^{-bc}\tau^{-cd}$. Putting these together gives
\[
	\alpha^{-1}\sigma^d\tau^{-c}\alpha=\sigma^{ad-bc}=\sigma.
\]
Hence $\alpha\sigma\alpha^{-1}=\sigma^d\tau^{-c}$. Similarly, we can write $\alpha\tau\alpha^{-1},\beta\sigma\beta^{-1}$, and $\beta\tau\beta^{-1}$ in terms of $\sigma$ and $\tau$. Also recall that $\alpha\beta=\sigma^m\tau^n\beta\alpha$.

Note that every element in $\pi_1M$ can be written in the form $\sigma^x\tau^y\beta^w\alpha^z$. Given some product of generators, we can use the above observations to move all the $\alpha$'s to the right (at the expense of picking up extra $\sigma$'s and $\tau$'s), then move all the $\beta$'s to the right.

Now suppose $M$ has a section. This induces a section $f:\pi_1S\to\pi_1M$ in the short exact sequence (\ref{eq:ses}). Using our above observations and the fact that $f$ is a section, we see that $f(\alpha)=\sigma^x\tau^y\alpha$ and $f(\beta)=\sigma^w\tau^z\beta$. Moreover, since $\alpha$ and $\beta$ commute in $\pi_1S$, $f(\alpha)$ and $f(\beta)$ must commute in $\pi_1M$. Thus
\begin{align*}
	&(\sigma^x\tau^y\alpha)(\sigma^w\tau^z\beta)=(\sigma^w\tau^z\beta)(\sigma^x\tau^y\alpha)\\
	&\sigma^x\tau^y(\sigma^d\tau^{-c})^w(\sigma^{-b}\tau^a)^z\alpha\beta=\sigma^w\tau^z(\sigma^s\tau^{-r})^x(\sigma^{-q}\tau^p)^y\beta\alpha\\
	&\alpha\beta\alpha^{-1}\beta^{-1}=\sigma^w\tau^z(\sigma^s\tau^{-r})^x(\sigma^{-q}\tau^p)^y\sigma^{-x}\tau^{-y}(\sigma^d\tau^{-c})^{-w}(\sigma^{-b}\tau^a)^{-z}\\
	&\sigma^m\tau^n=\sigma^w\tau^z(\sigma^s\tau^{-r})^x(\sigma^{-q}\tau^p)^y\sigma^{-x}\tau^{-y}(\sigma^d\tau^{-c})^{-w}(\sigma^{-b}\tau^a)^{-z}.
\end{align*}
This equation now only involves $\sigma$'s and $\tau$'s, and since $\pi_1F\to\pi_1M$ is an inclusion, we see that the powers of $\sigma$ and $\tau$ must match up. Writing $\binom{k}{l}$ to denote $\sigma^k\tau^l$, the above equation simplifies to
\begin{align*}
	\pvec{m}{n}&=\pvec{w}{z}+x\pvec{s}{-r}+y\pvec{-q}{p}-\pvec{x}{y}-w\pvec{d}{-c}-z\pvec{-b}{a}\\
	&=\begin{pmatrix}s & -q\\-r & p\end{pmatrix}\pvec{x}{y}-\pvec{x}{y}-\begin{pmatrix}d & -b\\-c & a\end{pmatrix}\pvec{w}{z}+\pvec{w}{z}\\
	&=(B^{-1}-I)v_1-(A^{-1}-I)v_2\\
	&=(B-I)v_1'+(A-I)v_2',
\end{align*}
where $v_1=\binom{x}{y}$, $v_2=\binom{w}{z}$, $v_1'=-B^{-1}v_1$, and $v_2'=-A^{-1}v_2$. Hence $\binom{m}{n}\in\spn{A-I,B-I}$.

Conversely, if $\binom{m}{n}\in\spn{A-I,B-I}$, then according to Proposition 2(3) of \cite{sf}, $M(A,B;m,n)$ is isomorphic to $M(A,B;0,0)$ as bundles, and $M(A,B;0,0)$ has a section by construction.
\end{proof}

\section{Symplectic Calabi-Yaus}
We now recall some facts about the symplectic Kodaira dimension, based on \cite{li}. Given a symplectic manifold $(X,\omega)$, let $(X_0,\omega_0)$ be a minimal model with canonical class $K_{X_0}$. The \textit{symplectic Kodaira dimension} $\kappa(X,\omega)$ is defined to be either $-\infty,0,1$, or $2$, depending on the signs of $K_{X_0}\cdot[\omega_0]$ and $K_{X_0}^2$. The definition of $\kappa$ is independent of the minimal model and of the particular symplectic form and is hence a smooth invariant.
\begin{dfn}
A symplectic 4-manifold $X$ is called a \textit{symplectic Calabi-Yau surface}, or SCY for short, if $\kappa(X)=0$. As observed by Li (\cite{li} Proposition 2.4), if $X_0$ is a minimal model of $X$ and $K_{X_0}$ its canonical class, then $\kappa=0$ is equivalent to saying that $K_{X_0}$ is torsion in $H^2(X_0;\Z)$.
\end{dfn}
\begin{rmk}
\label{rmk:scy}
Note that if $X$ is a symplectic 4-manifold with torsion canonical class $K$, then either $K$ is actually trivial or $X$ is an integral homology Enriques surface (in which case $K$ has order two); see Proposition 6.3 in \cite{li}. In this paper, all SCYs under consideration have $b_1\ge 2$ (since $b_1\ge 2$ for torus bundles; see \cite{sf}, Proposition 1), which rules out integral homology Enriques surfaces and their blowups, so our SCYs all have minimal models with $K=0$.

Observe that if $X$ is a SCY and $X=X_0\#n\overline\CP$, with $X_0$ minimal, then $0=c_1^2(X_0)=c_1^2(X)+n$. Hence if $X$ is a SCY and $c_1^2(X)=0$, then $X$ is already minimal. In this paper, all our SCYs will have $c_1^2=0$ by construction, and thus our SCYs are minimal. Combining this with our previous observation, all our SCYs have $K=0$. This will be implicit every time we invoke Corollary \ref{cor:homeo}.
\end{rmk}

The following gives a good source of SCYs:
\begin{prop}
(\cite{baykur}, Proposition 1; see also \cite{bh}, Theorem 4.1) Suppose $(X,f)$ is a genus-$g$ Lefschetz pencil with $b$ base points and $X$ is not rational or ruled. Then $X$ is a SCY or a rational homology Enriques surface if and only if $b=2g-2$.
\end{prop}
It turns out that a large class of SCYs are determined up to homeomorphism by their fundamental group. We recall the following from \cite{fv} (note that in their language, a SCY is required to have $K=0$, which excludes nonminimal 4-manifolds and integral homology Enriques surfaces):
\begin{thm}
(\cite{fv}, Corollary 3.3) Let $X$ be a SCY with $K_X=0$, and assume further that $X$ is an infrasolvmanifold. If $Y$ is a SCY with $K_Y=0$ and $\pi_1Y\cong\pi_1X$, then $X$ and $Y$ are homeomorphic.
\end{thm}
Note that all $T^2$-bundles over $T^2$ are SCYs with $K=0$ and are also infrasolvmanifolds \cite{li}, \cite{fv}. Thus we get the following corollary, which will tell us the homeomorphism types of the SCYs we construct:
\begin{cor}
\label{cor:homeo}
Let $X$ be a SCY with $K_X=0$. If $\pi_1X$ is isomorphic to $\pi_1M(A,B;m,n)$, then $X$ is homeomorphic to $M(A,B;m,n)$.
\end{cor}
\begin{rmk}
These arguments are not new. Baykur and Hamada-Hayano used this line of reasoning to identify homeomorphism types in \cite{baykur}, \cite{hh}.
\end{rmk}

\section{Semidirect products}
Given a matrix $A\in\GL(3,\Z)$, we get an action $\Z\curvearrowright\Z^3$, which in turn defines a semidirect product $\Z^3\rtimes_A\Z$.

\begin{prop}
\label{prop:torusbndl}
If $A=\begin{pmatrix} a & b\\c & d\end{pmatrix}\in\GL(2,\Z)$ and $M=M(A,I;m,n)$, then $\pi_1M\cong\Z^3\rtimes_{\tilde A}\Z$, where
\[
\tilde A=
\begin{pmatrix}
	a & b & -ma-nb\\
	c & d & -mc-nd\\
	0 & 0 & 1
\end{pmatrix}.
\]
\end{prop}
\begin{proof}
Using the presentation of $\pi_1M$ given in Proposition \ref{prop:fundgrp}, we see that $\sigma,\tau$, and $\beta$ all commute. Observe that $\alpha\beta\alpha^{-1}\beta^{-1}=\sigma^m\tau^n$ implies
\begin{align*}
	\alpha\beta\alpha^{-1}&=\sigma^m\tau^n\beta\\
	\beta&=\alpha^{-1}\sigma^m\tau^n\beta\alpha\\
	&=(\sigma^a\tau^c)^m(\sigma^b\tau^d)^n\alpha^{-1}\beta\alpha,
\end{align*}
and hence $\alpha^{-1}\beta\alpha=\sigma^{-ma-nb}\tau^{-mc-nd}\beta$. Thus $\pi_1M$ is isomorphic to $\Z^3\rtimes\Z$, where the $\Z^3$ subgroup is generated by $\sigma,\tau,\beta$ and the $\Z$-action by $\alpha$ is specified by
\[
	\alpha^{-1}\sigma\alpha=\sigma^a\tau^c,\;\;\alpha^{-1}\tau\alpha=\sigma^b\tau^d,\;\;\alpha^{-1}\beta\alpha=\sigma^{-ma-nb}\tau^{-mc-nd}\beta.
\]
\end{proof}
\begin{rmk}
If we are given $\tilde A$ as
\[\tilde A = \begin{pmatrix} a & b & x\\c & d & y\\0 & 0 & 1\end{pmatrix},\]
then we have $\Z^3\rtimes_{\tilde A}\Z\cong\pi_1 M(A,I;m,n)$, with $m,n$ given by
\[ \begin{pmatrix}m\\n\end{pmatrix} = -A^{-1}\begin{pmatrix}x\\y\end{pmatrix}.\]
\end{rmk}

\begin{lem}
\label{lem:conj}
Let $A,P\in\GL(3,\Z)$ and set $B=PAP^{-1}$. Then
\[ \Z^3\rtimes_A\Z\cong\Z^3\rtimes_B\Z.\]
\end{lem}
\begin{proof}
Define $\phi: \Z^3\rtimes_A\Z\to\Z^3\rtimes_B\Z$ by $\phi(v,n)=(Pv,n)$. In $\Z^3\rtimes_B\Z$, we have
\begin{equation*}
\begin{split}
	\phi(v,n)\phi(w,m)&=(Pv,n)(Pw,m)\\
	&=(Pv+B^nPw,n+m)\\
	&=(Pv+PA^nw,n+m)\\
	&=\phi(v+A^nw,n+m)\\
	&=\phi((v,n)(w,m)),
\end{split}
\end{equation*}
so $\phi$ is a homomorphism. $\phi^{-1}(v,n)=(P^{-1}v,n)$ is the inverse, so $\phi$ is an isomorphism.
\end{proof}

\section{Baykur's genus-3 pencils on SCYs}
In \cite{baykur}, Baykur constructed genus-3 pencils on SCYs by breeding two copies of Matsumoto's genus-2 relation. The genus-3 pencils depend on a mapping class $\phi$ in $\Mod(\Sigma_3^4)$ that fixes some set of curves on $\Sigma_3^4$ - this mapping class is used to conjugate one of the Matsumoto relations (the unprimed curves in Baykur's notation). When $\phi$ is the identity, the vanishing cycles are written in terms of the standard generators of $\pi_1(\Sigma_3)$ in \cite{baykur}, equations (48)-(59). A straightforward computation shows that the total space $X$ of the Lefschetz pencil has $\pi_1(X)=\Z^4$. Corollary \ref{cor:homeo} then implies that $X$ is homeomorphic to the 4-torus $T^4$ (the trivial $T^2$-bundle over $T^2$ with monodromy $(A,B)=(I,I)$ and obstruction class $(0,0)$).

By taking $\phi=t_{b_1}^lt_{a_3}^k$ (see \cite{baykur}, Figure 4 for these curves), Baykur obtains a family $X_{l,k}$ of SCYs that realizes all possible integral homology types of SCYs with $b_1>0$. As a warm-up for later computations, we will show that the $X_{l,k}$ are all homeomorphic to $T^2$-bundles over $T^2$. More precisely,
\begin{thm}
$X_{l,k}$ is homeomorphic to $M(T,I;0,0)$, where
\[
	T = \begin{pmatrix}lk+1 & l\\k & 1\end{pmatrix}.
\]
These torus bundles all have sections.
\end{thm}

\begin{proof}
 $\pi_1(X_{l,k})$ is obtained from $\pi_1(\Sigma_3)$ by dividing out all the vanishing cycles. We use the standard generators of $\pi_1(\Sigma_3)$, as in \cite{baykur}, Figure 4. We start with the primed curves coming from the unconjugated embedding of Matsumoto's relation, given in \cite{baykur}, equations (54)-(59):
\[ a_2^{-1}a_1a_2a_3=1 \tag{$B_0'$} \]
\[ a_1b_1^{-1}a_2a_3^2b_3^{-1}a_3^{-1}b_2a_2^{-1}=1 \tag{$B_1'$} \]
\[ b_1a_2b_2^{-1}a_3b_3a_3^{-1}a_2^{-1} = 1 \tag{$B_2'$} \]
\[ a_2^{-1}a_1a_2b_2^{-1}a_3b_2 = 1 \tag{$A_0'$} \]
\[a_1b_1^{-1}a_2b_2^{-1}a_3^2b_3^{-1}a_3^{-1}b_2^2a_2^{-1}=1 \tag{$A_1'$} \]
\[ b_1^{-1}a_2b_2^{-1}a_3b_3^{-1}a_3^{-1}b_2^2a_2^{-1}=1 \tag{$A_2'$} \]
\[ [a_1,b_1][a_2,b_2][a_3,b_3]=1 \tag{$\Sigma$} \]
(See \cite{baykur}, Figure 7 for these curves). Using ($B_0'$) in ($A_0'$), we see that $a_3^{-1}b_2^{-1}a_3b_2=1$, so $[a_3,b_2]=1$. Using this, we commute some $a_3$'s and $b_2$'s in ($A_1'$) and ($A_2'$):
\[ a_1b_1^{-1}a_2a_3^2b_2^{-1}b_3^{-1}a_3^{-1}b_2^2a_2^{-1}=1\]
\[ b_1^{-1}a_2a_3b_2^{-1}b_3^{-1}a_3^{-1}b_2^2a_2^{-1}=1\]
Cancelling off the $a_3b_2^{-1}b_3^{-1}a_3^{-1}b_2^2a_2^{-1}$ on the right side of both gives $a_1b_1^{-1}a_2a_3=b_1^{-1}a_2$. Thus $a_1b_1^{-1}=b_1^{-1}a_2a_3^{-1}a_2^{-1}$, and using ($B_0'$) we get $a_1b_1^{-1}=b_1^{-1}a_1$, so $[a_1,b_1]=1$.

We modify ($B_2'$) by first commuting $a_3$ and $b_2$, then using ($B_0'$), then commuting $a_1$ and $b_1$:
\begin{equation*}
\begin{split}
	1&=b_1a_2a_3b_2^{-1}b_3a_3^{-1}a_2^{-1}\\
	&=b_1a_1^{-1}a_2b_2^{-1}b_3a_2^{-1}a_1\\
	&=b_1a_2b_2^{-1}b_3a_2^{-1}.
\end{split}
\end{equation*}

So far, our relations simplify to
\[ a_1=a_2a_3^{-1}a_2^{-1} \tag{$B_0'$} \]
\[ a_1b_1^{-1}a_2a_3^2b_3^{-1}a_3^{-1}b_2a_2^{-1}=1 \tag{$B_1'$} \]
\[ b_1=a_2b_3^{-1}b_2a_2^{-1} \tag{$B_2'$} \]
\[ [a_3,b_2] = 1 \tag{$A_0'$} \]
\[ [a_1,b_1]=1 \tag{$A_1'$} \]
\[ b_1^{-1}a_2b_2^{-1}a_3b_3^{-1}a_3^{-1}b_2^2a_2^{-1}=1 \tag{$A_2'$} \]
\[ [a_2,b_2][a_3,b_3]=1 \tag{$\Sigma$} \]

Using ($B_0'$) and ($B_2'$) in ($A_1'$), we see that
\[ 1=[a_2a_3^{-1}a_2^{-1},a_2b_3^{-1}b_2a_2^{-1}]=a_2[a_3^{-1},b_3^{-1}b_2]a_2^{-1},\]
so $a_3$ commutes with $b_3^{-1}b_2$. Since $a_3$ already commutes with $b_2$ by ($A_0'$), it follows that $a_3$ commutes with $b_3$. Thus ($A_1'$) is equivalent to $[a_3,b_3]=1$ and ($\Sigma$) becomes $[a_2,b_2]=1$. We can modify ($B_1'$) using our commutators and ($B_0'$):
\begin{equation*}
\begin{split}
	1&=a_1b_1^{-1}a_2a_3b_3^{-1}b_2a_2^{-1}\\
	&=a_1b_1^{-1}a_1^{-1}a_2b_3^{-1}b_2a_2^{-1}\\
	&=b_1^{-1}a_2b_3^{-1}b_2a_2^{-1}.
\end{split}
\end{equation*}
This is the same as ($B_2'$), so we can eliminate ($B_1'$).
Lastly, we modify ($A_2'$) using our commutators and ($B_2'$):
\begin{equation*}
\begin{split}
	1&=a_2^{-1}b_1^{-1}a_2b_2^{-1}b_3^{-1}b_2^2\\
	&=b_2^{-1}b_3b_2^{-1}b_3^{-1}b_2^2,
\end{split}
\end{equation*}
so $b_2b_3b_2^{-1}b_3^{-1}=1$. That is, $[b_2,b_3]=1$. Our relations have been simplified to:
\[ a_1=a_2a_3^{-1}a_2^{-1} \tag{$B_0'$} \]
\[ b_1=a_2b_3^{-1}b_2a_2^{-1} \tag{$B_2'$} \]
\[ [a_3,b_2] = 1 \tag{$A_0'$} \]
\[ [a_3,b_3]=1 \tag{$A_1'$} \]
\[ [b_2,b_3]=1 \tag{$A_2'$} \]
\[ [a_2,b_2]=1 \tag{$\Sigma$} \]

We now determine the relations coming from the unprimed curves (the conjugated embedding of Matsumoto's relation). Again, see \cite{baykur}, Figure 7 for these curves. The induced action of $t_{b_1}$ on $\pi_1(\Sigma_3)$ sends $a_1$ to $a_1b_1^{-1}$ and fixes all other generators, and the induced action of $t_{a_3}$ on $\pi_1(\Sigma_3)$ sends $b_3$ to $b_3a_3$ and fixes all other generators. It follows that the induced action of $\phi=t_{b_1}^lt_{a_3}^k$ on $\pi_1(\Sigma_3)$ is given by
\[ \phi(a_1)=a_1b_1^{-l},\;\;\phi(b_3)=b_3a_3^k,\]
and all other generators are fixed. Combining this with (48)-(53) in \cite{baykur}, we get the additional relations in $\pi_1(X_{l,k})$:
\[ a_1b_1^{-l}a_3=1 \tag{$B_0$}\]
\[ a_1b_1^{-l-1}a_2b_2a_2^{-1}a_3^{-k+1}b_3^{-1}=1 \tag{$B_1$}\]
\[ b_1^{-1}a_2b_2a_2^{-1}a_3^{-k}b_3^{-1}=1 \tag{$B_2$}\]
\[ b_2^{-1}[a_3,b_3]a_1b_1^{-l}[b_3,a_3]b_2a_3=1 \tag{$A_0$}\]
\[ b_2^{-1}[a_3,b_3]a_1b_1^{-l-1}a_2b_2a_2^{-1}[b_3,a_3]b_2a_3^{-k+1}b_3^{-1}=1 \tag{$A_1$}\]
\[ b_2^{-1}[a_3,b_3]a_1b_1^{-1}a_1^{-1}([b_3,a_3]b_2)^2a_3^{-k}b_3^{-1}=1 \tag{$A_2$}\]
Using the commutators in ($A_0'$), ($A_1'$), ($A_2'$), and ($\Sigma$) gives:
\[ a_1b_1^{-l}a_3=1 \tag{$B_0$}\]
\[ a_1b_1^{-l-1}b_2a_3^{-k+1}b_3^{-1}=1 \tag{$B_1$}\]
\[ b_1^{-1}b_2a_3^{-k}b_3^{-1}=1 \tag{$B_2$}\]
\[ b_2^{-1}a_1b_1^{-l}b_2a_3=1 \tag{$A_0$}\]
\[ b_2^{-1}a_1b_1^{-l-1}b_2^2a_3^{-k+1}b_3^{-1}=1 \tag{$A_1$}\]
\[ b_2^{-1}b_1^{-1}b_2^2a_3^{-k}b_3^{-1}=1 \tag{$A_2$}\]
In ($A_1$) we can commute one power of $b_2$ to the very right using ($A_0'$) and ($\Sigma$), and this gives ($B_1$), so we can get rid of ($A_1$). Likewise, in ($A_2$) we can commute one power of $b_2$ to the very right and this gives ($B_2$), so we can get rid of ($A_2$). In ($A_0$) we can commute $b_2$ and $a_3$ and recognize this as ($B_0$), so ($A_0$) can be eliminated. Finally, note that ($B_0$), ($B_2$), and ($A_1'$) imply ($B_1$).

The only remaining relations are ($B_0$) and ($B_2$).Using ($B_0'$) and ($B_2'$) to eliminate the generators $a_1,b_1$ in ($B_0$) and ($B_2$) shows that $\pi_1(X_{l,k})$ is generated by $a_2,b_2,a_3,b_3$ with the relations
\[ a_2a_3^{-1}b_2^{-l}b_3^la_2^{-1}a_3=1 \tag{$B_0$}\]
\[ a_2b_2^{-1}b_3a_2^{-1}b_2a_3^{-k}b_3^{-1}=1 \tag{$B_2$}\]
\[ [a_3,b_2] = 1 \tag{$A_0'$} \]
\[ [a_3,b_3]=1 \tag{$A_1'$} \]
\[ [b_2,b_3]=1 \tag{$A_2'$} \]
\[ [a_2,b_2]=1 \tag{$\Sigma$} \]
($B_0$) can be rearranged to $a_2^{-1}a_3a_2=a_3b_2^lb_3^{-l}$. Cancelling the $b_2$'s in ($B_2$) and rearranging gives
\begin{equation*}
\begin{split}
	a_2^{-1}a_3^{-k}b_3^{-1}a_2&=b_3^{-1}\\
	(a_2^{-1}a_3a_2)^{-k}a_2^{-1}b_3^{-1}a_2&=b_3^{-1}\\
	a_3^{-k}b_2^{-lk}b_3^{lk}a_2^{-1}b_3^{-1}a_2&=b_3^{-1}\\
	a_2^{-1}b_3^{-1}a_2&=a_3^kb_2^{lk}b_3^{-lk-1}\\
	a_2^{-1}b_3a_2&=a_3^{-k}b_2^{-lk}b_3^{lk+1}
\end{split}
\end{equation*}
We are done simplifying the relations. The final presentation for $\pi_1(X_{l,k})$ is
\begin{multline}
	\pi_1(X_{l,k})=\langle a_2,b_2,a_3,b_3|a_2^{-1}a_3a_2=b_3^{-l}a_3b_2^l,a_2^{-1}b_3a_2=b_3^{lk+1}a_3^{-k}b_2^{-lk},\\
	[a_3,b_2]=1,[a_3,b_3]=1,[b_2,b_3]=1,[a_2,b_2]=1\rangle.
\end{multline}
This is precisely the presentation for a semidirect product $\Z^3\rtimes_A\Z$, with $\Z^3$ generated by the ordered basis $\{b_3,a_3,b_2\}$ and $\Z$ generated by $a_2$, and the action of $a_2$ on $\Z^3$ given by
\[ A=\begin{pmatrix} lk+1 & -l & 0\\-k & 1 & 0\\ -lk & l & 1\end{pmatrix}.\]
Set
\[ P=\begin{pmatrix} 1 & 0 & 0\\0 & -1 & 0\\1 & 0 & 1\end{pmatrix}\in\GL(3,\Z).\]
We get that
\[ B=PAP^{-1}=\begin{pmatrix} lk+1 & l & 0\\k & 1 & 0\\ 0 & 0 & 1\end{pmatrix}.\]
By Lemma \ref{lem:conj}, $\Z^3\rtimes_A \Z\cong\Z^3\rtimes_B \Z$, and by Proposition \ref{prop:torusbndl}, $\Z^3\rtimes_B\Z\cong\pi_1(M(T,I;0,0))$ with
\[ T = \begin{pmatrix} lk+1 & l\\k & 1\end{pmatrix}.\]
Baykur computed that $e(X_{l,k})=\sigma(X_{l,k})=0$, so $c_1^2(X_{l,k})=0$. Moreover, since we identified $\pi_1(X_{l,k})$ with that of a torus bundle, $b_1(X_{l,k})\ge 2$. It follows from Corollary \ref{cor:homeo} and Remark \ref{rmk:scy} that $X_{l,k}$ is homeomorphic to $M(T,I;0,0)$. These torus bundles all have sections by Proposition \ref{prop:section}.
\end{proof}

\section{More genus-3 pencils}
In this section, we will use the same breeding technique of \cite{baykur} to produce genus-3 pencils on a family of SCYs, and by identifying the resulting fundamental groups as those arising from torus bundles, we know the homeomorphism type of our SCYs. Moreover, many of the torus bundles in this family do not have sections.

Recall equation (5-2) in \cite{hh}, which is a modification of Matsumoto's $(6,2)$ relation in $\Sigma_2^4$. We embed $\Sigma_2^4$ in $\Sigma_3^4$ as in \cite{hh}, Figure 14 (d)-(f), giving the relation
\begin{equation}
\label{eq:hh}
	t_{d_7}t_{d_8}t_{d_9}t_{d_{10}}t_{d_{11}}t_{d_{12}}=t_{\delta_3}t_{\delta_4}t_{s_1}t_{s_2}t_{s_3}^{-1}t_{s_4}^{-1}.
\end{equation}
Also recall equation (44) in \cite{baykur}, which is another modification of Matsumoto's $(6,2)$ relation in $\Sigma_2^4$. We embed $\Sigma_2^4$ in $\Sigma_3^4$ as in \cite{baykur}, left half of Figure 7, giving the relation (45) in \cite{baykur}:
\begin{equation}
\label{eq:b}
	t_{B_0}t_{B_1}t_{B_2}t_{A_0}t_{A_1}t_{A_2}=t_{\partial_1}t_{\partial_2}t_{C_1'}t_{C_2'}t_{C_1}^{-1}t_{C_2}^{-1}.
\end{equation}
Combining the notation used in \cite{hh} and \cite{baykur}, we see that $\partial_1=\delta_1$, $\partial_2=\delta_2$, $C_1=s_2$, $C_2=s_1$, $C_1'=s_4$, and $C_2'=s_3$. Given any mapping class $\phi\in\Mod(\Sigma_3^4)$ that fixes $s_1,s_2,s_3$, and $s_4$, we can breed the relations (\ref{eq:hh}) and (\ref{eq:b}) to get the positive factorization
\begin{equation}
	t_{d_7}t_{d_8}t_{d_9}t_{d_{10}}t_{d_{11}}t_{d_{12}}t_{\phi(B_0)}t_{\phi(B_1)}t_{\phi(B_2)}t_{\phi(A_0)}t_{
	\phi(A_1)}t_{\phi(A_2)}=t_{\delta_1}t_{\delta_2}t_{\delta_3}t_{\delta_4}.
\end{equation}
Let us call the total space of the Lefschetz pencil associated to this factorization $X_\phi$.

We first determine the relations in $\pi_1$ coming from the vanishing cycles $d_7,\dots,d_{12}$. Writing them in terms of the standard generators of $\pi_1(\Sigma_3)$ (\cite{baykur}, Figure 4), we get the following relations:
\[ [b_3,a_3]a_2^{-1}a_1^{-1}a_2a_3=1 \tag{$d_7$}\]
\[ a_3b_3^{-1}a_3^{-1}a_2^{-1}b_1a_1^{-1}a_2a_3=1 \tag{$d_8$}\]
\[ b_3^{-1}a_3^{-1}a_2^{-1}b_1a_2a_3=1 \tag{$d_9$}\]
\[ a_3b_3^{-1}a_3^{-1}b_2a_2^{-1}a_1^{-1}a_2b_2^{-1}a_3b_3=1 \tag{$d_{10}$}\]
\[ a_3b_3^{-1}a_3^{-1}b_2a_2^{-1}b_1a_1^{-1}a_2b_2^{-1}a_3=1 \tag{$d_{11}$}\]
\[ b_3^{-1}a_3^{-1}b_2a_2^{-1}b_1a_2b_2^{-1}a_3=1 \tag{$d_{12}$}\]
In ($d_9$), moving the $a_3$ to the front and comparing this to ($d_8$), we see that
\[ a_1^{-1}a_2a_3=a_2. \tag{$d_8'$}\]
Let's replace ($d_8$) with ($d_8'$), keeping all other relations the same. Using ($d_8'$), we see that ($d_7$) becomes $[b_3,a_3]=1$. Now using $[b_3,a_3]=1$, we can cancel the $a_3$'s in ($d_9$) to get $b_3^{-1}a_2^{-1}b_1a_2=1$. So far, relations ($d_{10}$)-($d_{12}$) remain unchanged and relations ($d_7$)-($d_9$) now read
\[ [a_3,b_3]=1 \tag{$d_7$}\]
\[ a_1^{-1}a_2a_3=a_2 \tag{$d_8'$}\]
\[ b_3^{-1}a_2^{-1}b_1a_2=1 \tag{$d_9$}\]
Using $[a_3,b_3]=1$ in ($d_{10}$), we get $b_2a_2^{-1}a_1^{-1}a_2b_2^{-1}a_3=1$. Now use ($d_8'$) to see that $b_2a_3^{-1}b_2^{-1}a_3=1$. Thus ($d_{10}$) has become $[b_2,a_3]=1$. Now we work on ($d_{11}$): use $[a_3,b_3]=1$ and $[b_2,a_3]=1$ to cancel and commute the $a_3$'s: $b_3^{-1}b_2a_2^{-1}b_1a_1^{-1}a_2a_3b_2^{-1}=1$. Use ($d_8'$) to rewrite this as $b_3^{-1}b_2a_2^{-1}b_1a_2b_2^{-1}=1$. Finally, use ($d_9$) to write this as $b_3^{-1}b_2b_3b_2^{-1}=1$. At this point, we have simplified ($d_{10}$) and ($d_{11}$) to
\[ [b_2,a_3]=1 \tag{$d_{10}$}\]
\[ [b_2,b_3]=1 \tag{$d_{11}$}\]
It turns out ($d_{12}$) is redundant: using $[a_3,b_3]=1$ to cancel the $a_3$'s gives $b_3^{-1}b_2a_2^{-1}b_1a_2b_2^{-1}=1$. Now using $[b_2,b_3]=1$ to cancel the $b_2$'s gives ($d_9$).

Rewrite ($d_8'$) and ($d_9$) as $a_1=a_2a_3a_2^{-1}$, $b_1=a_2b_3a_2^{-1}$ to see that $[a_1,b_1]=1$. The surface relation in $\pi_1(\Sigma_3)$ then collapses to $[a_2,b_2]=1$. In summary, we have the following relations coming from the vanishing cycles $d_7,\dots,d_{12}$ and the surface relation:
\begin{equation}
\label{eq:hh-rels}
	a_1=a_2a_3a_2^{-1},\; b_1=a_2b_3a_2^{-1},\; [a_3,b_3]=[b_2,a_3]=[b_2,b_3]=[a_2,b_2]=1.
\end{equation}

With some simplifications from (\ref{eq:hh-rels}), the vanishing cycles $B_0,\dots,A_2$ contribute these relations (see also \cite{baykur}, equations (48)-(53)):
\[ a_1a_3=1 \tag{$B_0$}\]
\[ a_1b_1^{-1}b_2a_3b_3^{-1}=1 \tag{$B_1$} \]
\[ b_1^{-1}b_2b_3^{-1}=1 \tag{$B_2$} \]
\[ a_1a_3=1 \tag{$A_0$}\]
\[ a_1b_1^{-1}b_2a_3b_3^{-1}=1 \tag{$A_1$}\]
\[ b_1^{-1}b_2b_3^{-1}=1 \tag{$A_2$}\]
Note that the relations coming from $A_0,A_1,A_2$ are redundant, and that relation ($B_1$) follows from ($B_0$), ($B_2$), and the fact that $[a_3,b_3]=1$. 

Fix some $C=\begin{pmatrix} p & q\\r & s\end{pmatrix}\in\SL(2,\Z)$. Since $a_1$ and $b_1$ commute, we can let $\phi\in\Mod(\Sigma_3^4)$ be an appropriate product of $t_{a_1}$'s and $t_{b_1}$'s so that $\phi(a_1)=a_1^pb_1^r$, $\phi(b_1)=a_1^qb_1^s$, and $\phi$ fixes the other standard generators of $\pi_1(\Sigma_3)$. Hence the only two new relations coming from $\phi(B_0),\dots,\phi(A_2)$ are:
\[ a_1^pb_1^ra_3=1 \tag{$\phi(B_0)$}\]
\[ (a_1^qb_1^s)^{-1}b_2b_3^{-1}=1 \tag{$\phi(B_2)$}\]
In summary, $\pi_1(X_\phi)$ is the group generated by $a_1,b_1,a_2,b_2,a_3,b_3$ with relations (\ref{eq:hh-rels}), ($\phi(B_0)$), and ($\phi(B_2)$). Using the first two relations in (\ref{eq:hh-rels}) to eliminate $a_3$ and $b_3$, we get the following presentation of $\pi_1(X_\phi)$:
\begin{multline}
	\pi_1(X_\phi)\cong\langle a_1,b_1,a_2,b_2|a_2^{-1}a_1a_2=a_1^{-p}b_1^{-r},a_2^{-1}b_1a_2=a_1^{-q}b_1^{-s}b_2,\\
	[a_1,b_1]=[a_1,b_2]=[b_1,b_2]=[a_2,b_2]=1\rangle.
\end{multline}

It is clear from this presentation that $\pi_1(X_\phi)$ is isomorphic to $\Z^3\rtimes_A\Z$, with $\Z^3$ generated by $\{a_1,b_1,b_2\}$, $\Z$ generated by $a_2$, and the action of $a_2$ on $\Z^3$ given by
\[
	A = \begin{pmatrix}
	-p & -q & 0\\
	-r & -s & 0\\
	0 & 1 & 1
	\end{pmatrix}.
\]
By computations identical to those in \cite{baykur}, Section 5.1, we find that $e(X_\phi)=\sigma(X_\phi)=0$, so $c_1^2(X_\phi)=0$. By our identifications of $\pi_1(X_\phi)$ with those of torus bundles below, we have that $b_1(X_\phi)\ge2$. Hence Remark \ref{rmk:scy} applies and we can use Corollary \ref{cor:homeo} to determine the homeomorphism type of $X_\phi$.

It is straightforward, although a bit tedious, to do a change of basis that makes $A$ look standard as in Proposition \ref{prop:torusbndl}.

Suppose first that $p=-1$ and $r=0$. Since $C\in\SL(2,\Z)$, we see that $s=-1$ and $C=\begin{pmatrix} -1 & q\\0 & -1\end{pmatrix}$. Thus
\[ A = \begin{pmatrix} 1 & -q & 0\\ 0 & 1 & 0\\ 0 & 1 & 1\end{pmatrix}.\]
Set
\[ P = \begin{pmatrix} 0 & 1 & 0\\ 0 & 0 & 1\\ 1 & 0 & q\end{pmatrix}\in\GL(3,\Z).\]
Then
\[ PAP^{-1}=\begin{pmatrix} 1 & 0 & 0\\1 & 1 & 0\\0 & 0 & 1\end{pmatrix},\]
so by Proposition \ref{prop:torusbndl} and Lemma \ref{lem:conj}, $X_\phi$ is homeomorphic to $M(T,I;0,0)$ for $T=\begin{pmatrix} 1 & 0\\1 & 1\end{pmatrix}$. Note that this torus bundle is independent of $q$ and has a section.

Now suppose that $(p+1,r)\ne(0,0)$. Let $d=\gcd(r,p+1)$, and let $x,y\in\Z$ be such that $x\frac{p+1}{d}+y\frac{r}{d}=1$. Set
\[ P=\begin{pmatrix} x & y & 0\\
0 & 0 & 1\\ -\frac{r}{d} & \frac{p+1}{d} & \frac{p+s+2}{d}\end{pmatrix}\in\GL(3,\Z).\]
Then
\[
	P^{-1}=\begin{pmatrix} \frac{p+1}{d} & y\frac{p+s+2}{d} & -y\\ \frac{r}{d} & -x\frac{p+s+2}{d} & x\\ 0 & 1 & 0\end{pmatrix},
\]
and hence
\begin{multline*}
	PAP^{-1}=\\
	\begin{pmatrix}
	\frac{p+1}{d}(x-d)-\frac{r}{d}(qx+sy) & \frac{p+s+2}{d}[y(x-d)+x(qx+sy)] & -y(x-d)-x(qx+sy)\\
	\frac{r}{d} & 1-x\frac{p+s+2}{d} & x\\
	0 & 0 & 1
	\end{pmatrix}.
\end{multline*}
This shows, using Proposition \ref{prop:torusbndl} and Lemma \ref{lem:conj}, that $X_\phi$ is homeomorphic to $M(T,I;m,n)$, where
\[ 
	T=\begin{pmatrix}
	\frac{p+1}{d}(x-d)-\frac{r}{d}(qx+sy) & \frac{p+s+2}{d}[y(x-d)+x(qx+sy)]\\
	\frac{r}{d} & 1-x\frac{p+s+2}{d}
	\end{pmatrix}
\]
and $(m,n)=(y(x-d)+x(qx+sy),d-x)$.

Observe that Proposition \ref{prop:section} implies that $M(T,I;m,n)$ has a section if and only if $\binom{m}{n}$ is in the $\Z$-span of the columns of $T-I$. We have that
\begin{equation*}
\begin{split}
	(p+1)(x-d)-r(qx+sy)-d&=x(p+1-rq)-rsy-d(p+1)-d\\
	&=x(p+1-rq)-rsy-d(p+2)\\
	&=x(p+1+1-ps)-rsy-d(p+2)\\
	&=x(p+2)-xps-rsy-d(p+2)\\
	&=(x-d)(p+2)-s(xp+ry)\\
	&=(x-d)(p+2)-s(d-x)\\
	&=(x-d)(p+s+2),
\end{split}
\end{equation*}
which shows that the columns of $T-I$ are
\[ v=\begin{pmatrix}\frac{p+s+2}{d}(x-d)\\ \frac{r}{d}\end{pmatrix},w=\frac{p+s+2}{d}\begin{pmatrix}y(x-d)+x(qx+sy)\\ -x\end{pmatrix}.\]
Write $v_1=\frac{p+s+2}{d}(x-d)$ and $w_1=y(x-d)+x(qx+sy)$ for convenience, so
\[ v=\begin{pmatrix}v_1\\ \frac{r}{d}\end{pmatrix}, w=\frac{p+s+2}{d}\begin{pmatrix}w_1\\-x\end{pmatrix},\begin{pmatrix}m\\n\end{pmatrix}=\begin{pmatrix}w_1\\d-x\end{pmatrix}.\]

Note that
\begin{equation}
\label{eq:alg}
\begin{split}
	rw_1&=ry(x-d)+rx(qx+sy)\\
	&=(x-d)(d-x(p+1))+rqx^2+xs(d-x(p+1))\\
	&=-d^2+d(x(p+1)+x)-x^2(p+1)+rqx^2+dsx-x^2s(p+1)\\
	&=-d^2+dx(p+2+s)-x^2(p+1-rq+sp+s)\\
	&=-d^2+dx(p+s+2)-x^2(p+s+2)\\
	&=-d^2-x(x-d)(p+s+2).
\end{split}
\end{equation}

Since $\det(T-\lambda I)=\lambda^2-(\text{tr }T)\lambda+1$, we see that
\[\det(T-I)=2-\text{tr }T=-\text{tr}(T-I)=p+s+2.\]
If $p+s+2\ne 0$, then $\binom{m}{n}$ is of course in the $\Q$-span of $v,w$: writing $av+bw=\binom{m}{n}$, we see that
\begin{equation*}
\begin{split}
	a&=\frac{1}{p+s+2}(w_1(-x\frac{p+s+2}{d})-(d-x)(\frac{p+s+2}{d}w_1))\\
	&=-\frac{w_1x}{d}-\frac{w_1(d-x)}{d}=-w_1\in\Z\\
	b&=\frac{1}{p+s+2}(-\frac{r}{d}w_1+\frac{p+s+2}{d}(x-d)(d-x))\\
	&=\frac{1}{d(p+s+2)}(-rw_1-(p+s+2)(x-d)^2)\\
	&=\frac{1}{d(p+s+2)}(d^2+x(x-d)(p+s+2)-(x-d)^2(p+s+2))\\
	&=\frac{1}{d(p+s+2)}(d^2+d(x-d)(p+s+2))\\
	&=x-d+\frac{d}{p+s+2},
\end{split}
\end{equation*}
so $a,b\in\Z$ if and only if $p+s+2$ divides $d$.

If $p+s+2=0$, then
\[ v=\begin{pmatrix}0\\ \frac{r}{d}\end{pmatrix},w=\begin{pmatrix}0\\0\end{pmatrix},\]
and (\ref{eq:alg}) implies $rz_1=-d^2\ne 0$, so $z_1\ne 0$, and thus $\binom{m}{n}=\binom{z_1}{d-x}$ is not in the $\Z$-span of $v,w$.

To summarize, we have shown:
\begin{thm}
Let $C=\begin{pmatrix}p & q\\r & s\end{pmatrix}\in\SL(2,\Z)$ and suppose $(p,r)\ne(-1,0)$. Let $d=\gcd(r,p+1)$, and let $x,y\in\Z$ be such that $x(p+1)+yr=d$. Then there is a genus-3 Lefschetz pencil with total space homeomorphic to $M(T,I;m,n)$, with
\[ 
	T=\begin{pmatrix}
	\frac{p+1}{d}(x-d)-\frac{r}{d}(qx+sy) & \frac{p+s+2}{d}[y(x-d)+x(qx+sy)]\\
	\frac{r}{d} & 1-x\frac{p+s+2}{d}
	\end{pmatrix}
\]
and $(m,n)=(y(x-d)+x(qx+sy),d-x)$. This torus bundle has a section if and only if $p+s+2$ is nonzero and divides $d$.
\end{thm}
\begin{rmk}
The case where $(p,r)=(-1,0)$ produced a torus bundle with a section, which is not new since Hamada-Hayano \cite{hh} constructed, for any commuting pair $A,B\in\SL(2,\Z)$, a genus-3 Lefschetz pencil with total space homeomorphic to $M(A,B;0,0)$.
\end{rmk}

\section{Genus-3 pencils from genus-2 relations of type $(4,3)$}
\label{sec:breed43}
The previous pencils were obtained by breeding two copies of a genus-2 relation of type $(6,2)$. It is also possible to obtain genus-3 SCY pencils by breeding relations of type $(4,3)$. We compute some examples in this section, although it is not clear how to systematically carry out the fundamental group computations like we did in the previous section.

Our first example uses Hamada's (4,3) relation (see \cite{nakamura}, Theorem 3), which was originally written down as a genus-2 Lefschetz fibration but has a lift to $\Mod(\Sigma_2^3)$, giving a genus-2 pencil with 3 basepoints \cite{huang}:
\begin{equation}
\label{eq:hamada}
	t_\sigma t_Et_\gamma t_F t_Gt_\alpha t_D = t_{d_1}t_{d_2}t_{d_3}.
\end{equation}
See Figure \ref{fig:hamada} for the vanishing cycles. 
\begin{figure}[h]
\centering{
\resizebox{130mm}{!}{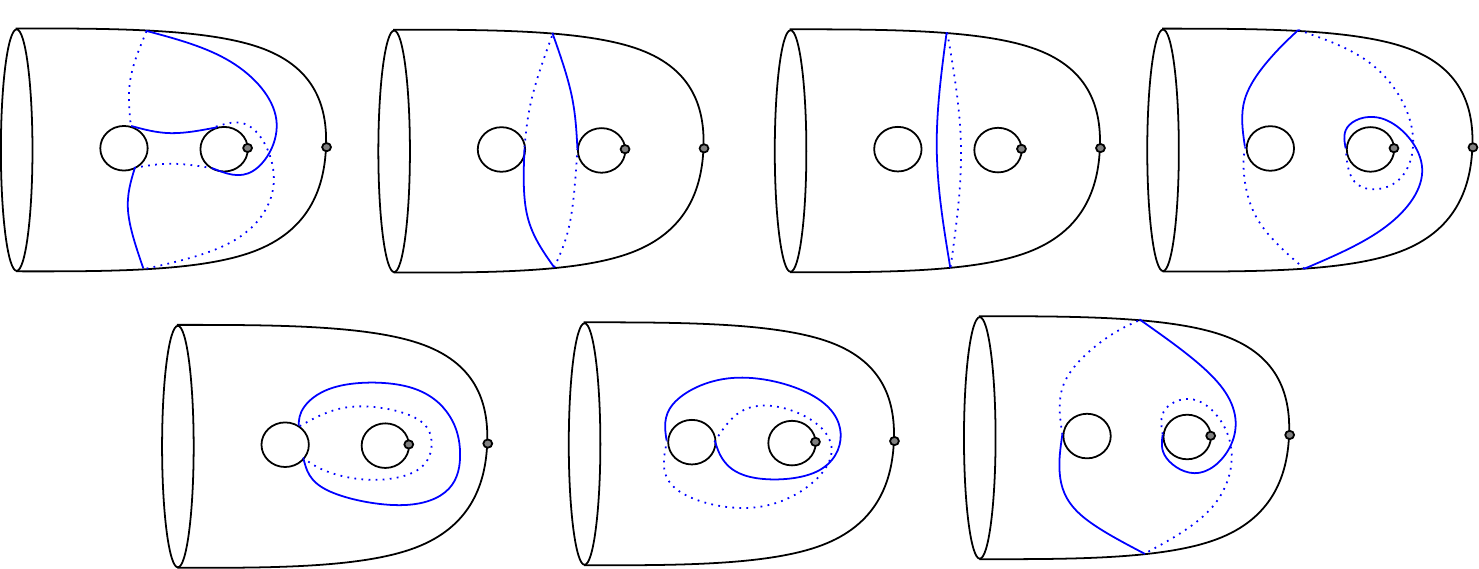}
\caption{The vanishing cycles of Hamada's genus-2 pencil.}
\label{fig:hamada}
}
\end{figure}
\begin{figure}[h]
\centering{
\resizebox{90mm}{!}{
\begingroup%
  \makeatletter%
  \providecommand\color[2][]{%
    \errmessage{(Inkscape) Color is used for the text in Inkscape, but the package 'color.sty' is not loaded}%
    \renewcommand\color[2][]{}%
  }%
  \providecommand\transparent[1]{%
    \errmessage{(Inkscape) Transparency is used (non-zero) for the text in Inkscape, but the package 'transparent.sty' is not loaded}%
    \renewcommand\transparent[1]{}%
  }%
  \providecommand\rotatebox[2]{#2}%
  \newcommand*\fsize{\dimexpr\f@size pt\relax}%
  \newcommand*\lineheight[1]{\fontsize{\fsize}{#1\fsize}\selectfont}%
  \ifx\svgwidth\undefined%
    \setlength{\unitlength}{406.76572785bp}%
    \ifx\svgscale\undefined%
      \relax%
    \else%
      \setlength{\unitlength}{\unitlength * \real{\svgscale}}%
    \fi%
  \else%
    \setlength{\unitlength}{\svgwidth}%
  \fi%
  \global\let\svgwidth\undefined%
  \global\let\svgscale\undefined%
  \makeatother%
  \begin{picture}(1,0.48043508)%
    \lineheight{1}%
    \setlength\tabcolsep{0pt}%
    \put(0,0){\includegraphics[width=\unitlength,page=1]{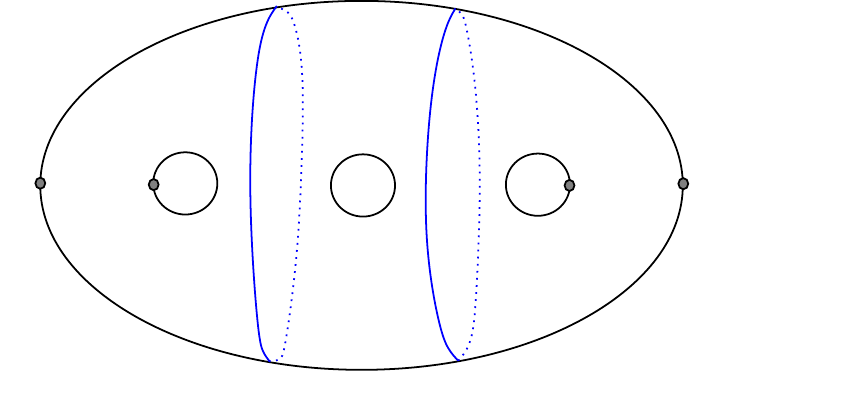}}%
    \put(0.28197964,0.00711217){\color[rgb]{0,0,0}\makebox(0,0)[lt]{\lineheight{1.25}\smash{\begin{tabular}[t]{l}$C$\end{tabular}}}}%
    \put(0.51219127,0.00711217){\color[rgb]{0,0,0}\makebox(0,0)[lt]{\lineheight{1.25}\smash{\begin{tabular}[t]{l}$C'$\end{tabular}}}}%
    \put(-0.00155956,0.28109554){\color[rgb]{0,0,0}\makebox(0,0)[lt]{\lineheight{1.25}\smash{\begin{tabular}[t]{l}$\delta_1$\end{tabular}}}}%
    \put(0.14560456,0.22432573){\color[rgb]{0,0,0}\makebox(0,0)[lt]{\lineheight{1.25}\smash{\begin{tabular}[t]{l}$\delta_3$\end{tabular}}}}%
    \put(0.66892173,0.28449313){\color[rgb]{0,0,0}\makebox(0,0)[lt]{\lineheight{1.25}\smash{\begin{tabular}[t]{l}$\delta_4$\end{tabular}}}}%
    \put(0.81256353,0.22526636){\color[rgb]{0,0,0}\makebox(0,0)[lt]{\lineheight{1.25}\smash{\begin{tabular}[t]{l}$\delta_2$\end{tabular}}}}%
  \end{picture}%
\endgroup%
}
\caption{Curves specifying embeddings $\Sigma_2^3\hookrightarrow \Sigma_3^4$.}
\label{fig:genus3}
}
\end{figure}

Figure \ref{fig:genus3} describes two embeddings of $\Sigma_2^3$ into $\Sigma_3^4$: one embedding sends $d_1\mapsto C, \sigma\mapsto C', d_2\mapsto\delta_4$, and $d_3\mapsto\delta_2$; the other embedding sends $d_1\mapsto C', \sigma\mapsto C, d_2\mapsto\delta_3$, and $d_3\mapsto\delta_1$. Under these embeddings, the relation (\ref{eq:hamada}) becomes
\begin{equation}
\label{eq:hamadaright}
	t_C t_Et_\gamma t_F t_Gt_\alpha t_D = t_{C'}t_{\delta_4}t_{\delta_2},
\end{equation}
\begin{equation}
\label{eq:hamadaleft}
	t_{C'} t_{E'}t_{\gamma'} t_{F'} t_{G'}t_{\alpha'} t_{D'} = t_{C}t_{\delta_3}t_{\delta_1}.
\end{equation}
Se Figures \ref{fig:unprimed} and \ref{fig:primed} for these curves.
\begin{figure}[h]
\centering{
\resizebox{130mm}{!}{
\begingroup%
  \makeatletter%
  \providecommand\color[2][]{%
    \errmessage{(Inkscape) Color is used for the text in Inkscape, but the package 'color.sty' is not loaded}%
    \renewcommand\color[2][]{}%
  }%
  \providecommand\transparent[1]{%
    \errmessage{(Inkscape) Transparency is used (non-zero) for the text in Inkscape, but the package 'transparent.sty' is not loaded}%
    \renewcommand\transparent[1]{}%
  }%
  \providecommand\rotatebox[2]{#2}%
  \newcommand*\fsize{\dimexpr\f@size pt\relax}%
  \newcommand*\lineheight[1]{\fontsize{\fsize}{#1\fsize}\selectfont}%
  \ifx\svgwidth\undefined%
    \setlength{\unitlength}{756.05401467bp}%
    \ifx\svgscale\undefined%
      \relax%
    \else%
      \setlength{\unitlength}{\unitlength * \real{\svgscale}}%
    \fi%
  \else%
    \setlength{\unitlength}{\svgwidth}%
  \fi%
  \global\let\svgwidth\undefined%
  \global\let\svgscale\undefined%
  \makeatother%
  \begin{picture}(1,0.45871053)%
    \lineheight{1}%
    \setlength\tabcolsep{0pt}%
    \put(0,0){\includegraphics[width=\unitlength,page=1]{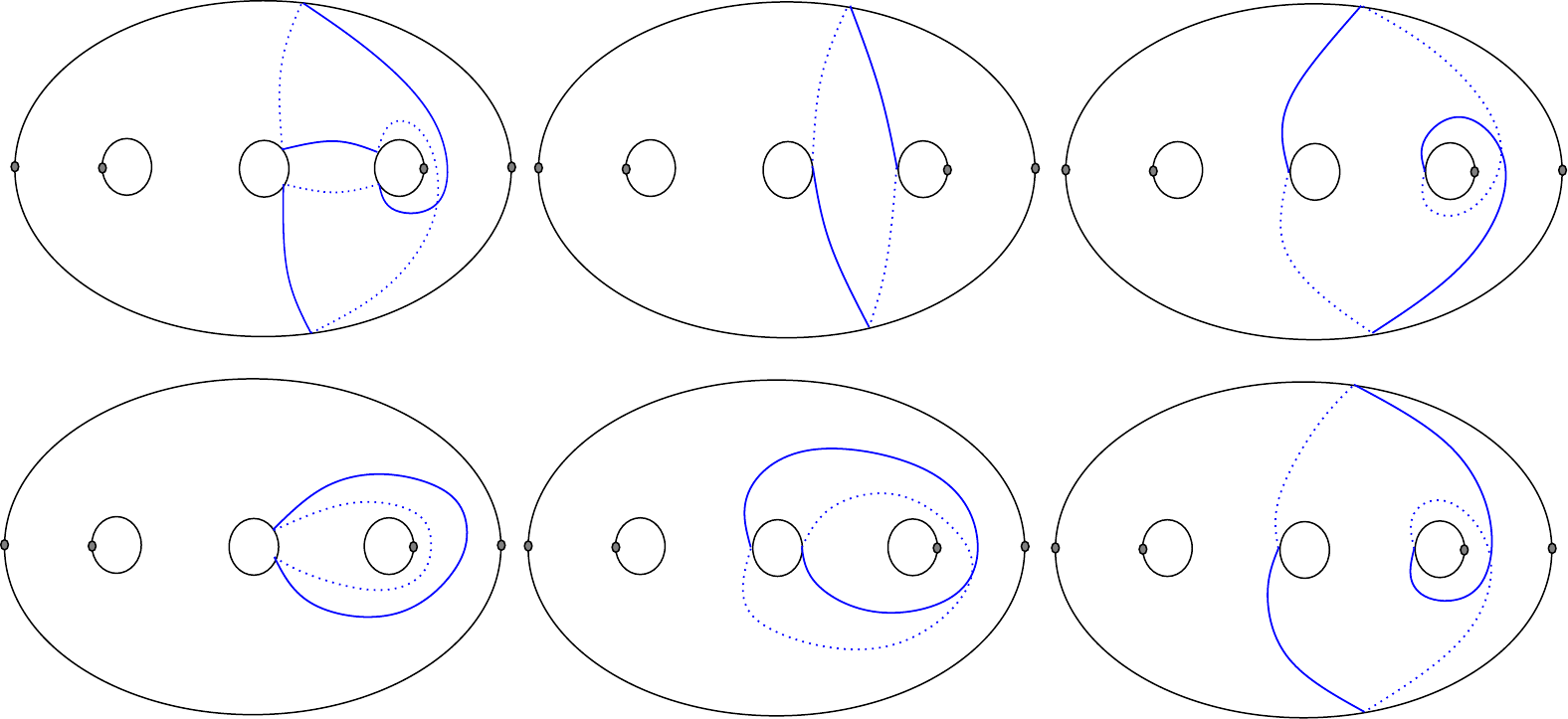}}%
    \put(0.07455854,0.30010861){\color[rgb]{0,0,0}\makebox(0,0)[lt]{\lineheight{1.25}\smash{\begin{tabular}[t]{l}$\alpha$\end{tabular}}}}%
    \put(0.40578719,0.29964555){\color[rgb]{0,0,0}\makebox(0,0)[lt]{\lineheight{1.25}\smash{\begin{tabular}[t]{l}$D$\end{tabular}}}}%
    \put(0.72427494,0.30672319){\color[rgb]{0,0,0}\makebox(0,0)[lt]{\lineheight{1.25}\smash{\begin{tabular}[t]{l}$E$\end{tabular}}}}%
    \put(0.08670851,0.05075422){\color[rgb]{0,0,0}\makebox(0,0)[lt]{\lineheight{1.25}\smash{\begin{tabular}[t]{l}$\gamma$\end{tabular}}}}%
    \put(0.38809214,0.05900985){\color[rgb]{0,0,0}\makebox(0,0)[lt]{\lineheight{1.25}\smash{\begin{tabular}[t]{l}$F$\end{tabular}}}}%
    \put(0.7266348,0.05900985){\color[rgb]{0,0,0}\makebox(0,0)[lt]{\lineheight{1.25}\smash{\begin{tabular}[t]{l}$G$\end{tabular}}}}%
  \end{picture}%
\endgroup%
}
\caption{Unprimed curves in our pencil (\ref{eq:hamadabreed}).}
\label{fig:unprimed}
}
\end{figure}
\begin{figure}[h]
\centering{
\resizebox{130mm}{!}{
\begingroup%
  \makeatletter%
  \providecommand\color[2][]{%
    \errmessage{(Inkscape) Color is used for the text in Inkscape, but the package 'color.sty' is not loaded}%
    \renewcommand\color[2][]{}%
  }%
  \providecommand\transparent[1]{%
    \errmessage{(Inkscape) Transparency is used (non-zero) for the text in Inkscape, but the package 'transparent.sty' is not loaded}%
    \renewcommand\transparent[1]{}%
  }%
  \providecommand\rotatebox[2]{#2}%
  \newcommand*\fsize{\dimexpr\f@size pt\relax}%
  \newcommand*\lineheight[1]{\fontsize{\fsize}{#1\fsize}\selectfont}%
  \ifx\svgwidth\undefined%
    \setlength{\unitlength}{756.05401467bp}%
    \ifx\svgscale\undefined%
      \relax%
    \else%
      \setlength{\unitlength}{\unitlength * \real{\svgscale}}%
    \fi%
  \else%
    \setlength{\unitlength}{\svgwidth}%
  \fi%
  \global\let\svgwidth\undefined%
  \global\let\svgscale\undefined%
  \makeatother%
  \begin{picture}(1,0.45871053)%
    \lineheight{1}%
    \setlength\tabcolsep{0pt}%
    \put(0,0){\includegraphics[width=\unitlength,page=1]{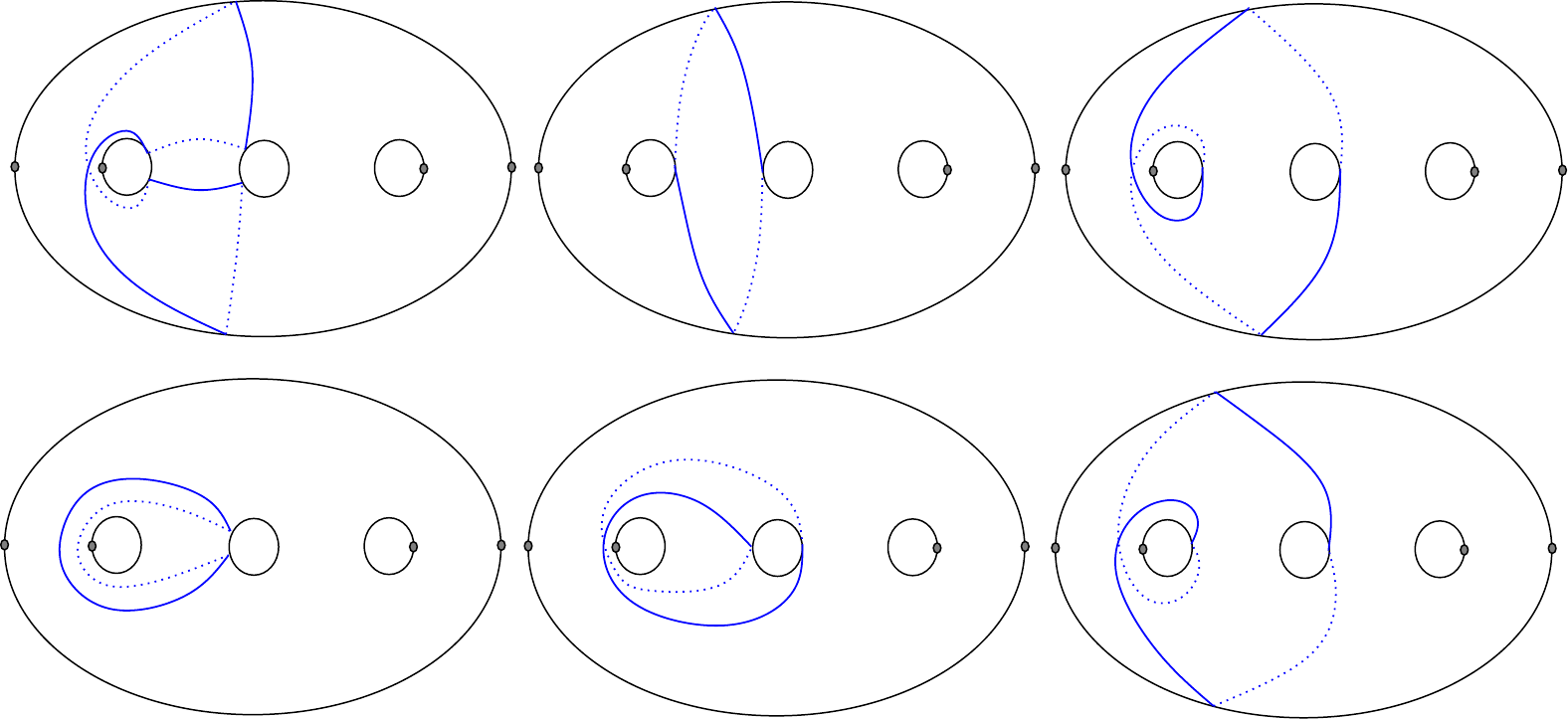}}%
    \put(0.17256431,0.29946944){\color[rgb]{0,0,0}\makebox(0,0)[lt]{\lineheight{1.25}\smash{\begin{tabular}[t]{l}$\alpha'$\end{tabular}}}}%
    \put(0.5099561,0.29613309){\color[rgb]{0,0,0}\makebox(0,0)[lt]{\lineheight{1.25}\smash{\begin{tabular}[t]{l}$D'$\end{tabular}}}}%
    \put(0.87111891,0.29988794){\color[rgb]{0,0,0}\makebox(0,0)[lt]{\lineheight{1.25}\smash{\begin{tabular}[t]{l}$E'$\end{tabular}}}}%
    \put(0.13836769,0.05424577){\color[rgb]{0,0,0}\makebox(0,0)[lt]{\lineheight{1.25}\smash{\begin{tabular}[t]{l}$\gamma'$\end{tabular}}}}%
    \put(0.51579569,0.05382921){\color[rgb]{0,0,0}\makebox(0,0)[lt]{\lineheight{1.25}\smash{\begin{tabular}[t]{l}$F'$\end{tabular}}}}%
    \put(0.87904323,0.06133503){\color[rgb]{0,0,0}\makebox(0,0)[lt]{\lineheight{1.25}\smash{\begin{tabular}[t]{l}$G'$\end{tabular}}}}%
  \end{picture}%
\endgroup%
}
\caption{Primed curves in our pencil (\ref{eq:hamadabreed}).}
\label{fig:primed}
}
\end{figure}

For any mapping class $\phi\in\Mod(\Sigma_3^4)$ that fixes $C'$ and $C$, we may conjugate (\ref{eq:hamadaright}) by $\phi$ to obtain
\[
	t_C t_{\phi(E)}t_{\phi(\gamma)} t_{\phi(F)} t_{\phi(G)}t_{\phi(\alpha)} t_{\phi(D)} = t_{C'}t_{\delta_4}t_{\delta_2}.
\]
Moving $t_C$ to the right in the above equation, and moving $t_{C'}$ to the right in (\ref{eq:hamadaleft}), we may breed to obtain the following genus-3 Lefschetz pencil with 4-basepoints.
\begin{equation}
\label{eq:hamadabreed}
	t_{E'}t_{\gamma'} t_{F'} t_{G'}t_{\alpha'} t_{D'}t_{\phi(E)}t_{\phi(\gamma)} t_{\phi(F)} t_{\phi(G)}t_{\phi(\alpha)} t_{\phi(D)}=t_{\delta_1}t_{\delta_2}t_{\delta_3}t_{\delta_4}
\end{equation}
Note that the total space of the Lefschetz pencil has $e=\sigma=0$, and hence $c_1^2=0$. So Remark \ref{rmk:scy} applies.

\begin{prop}
Let $\phi=t_{b_2}^k$ in the pencil (\ref{eq:hamadabreed}), and denote the total space by $X_k$.
\begin{enumerate}
\item If $k=2t$ is even then $X_k$ is homeomorphic to $M(T,I;1,0)$, where
\[
	T = \begin{pmatrix}-1 & -t\\0 & -1\end{pmatrix}.
\]
\item If $k=2t+1$ is odd then $X_k$ is homeomorphic to $M(T,I;-t,1)$, where
\[
	T = \begin{pmatrix}-1 & -k\\0 & -1\end{pmatrix}.
\]
\end{enumerate}
None of these torus bundles have sections.
\end{prop}
\begin{proof}
Writing the primed curves in terms of the standard basis for $\pi_1(\Sigma_3)$, we get the following relations in $\pi_1(X_k)$:
\[ a_1b_1^{-1}a_2b_2^{-1}a_2^{-1}b_1a_1^{-1}[b_3,a_3]b_2=1 \tag{$\alpha'$} \]
\[ [b_3,a_3]b_2b_1=1 \tag{$D'$} \]
\[ a_1b_1a_1a_2b_2=1 \tag{$E'$} \]
\[ a_1a_2b_2^{-1}a_2^{-1}b_1a_1^{-1}b_1^{-1}a_2b_2a_2^{-1}=1 \tag{$\gamma'$} \]
\[ a_1a_2b_2^{-1}a_2^{-1}b_1a_1b_1^{-1}a_2b_2=1 \tag{$F'$} \]
\[a_1^2b_1^{-1}a_2b_2^{-1}[a_3,b_3]=1 \tag{$G'$} \]
Using ($D'$), we see that $b_1^{-1}=[b_3,a_3]b_2$. This allows us to simplify ($\alpha'$) and ($G'$) by getting rid of the commutators:
\[ a_1b_1^{-1}a_2b_2^{-1}a_2^{-1}b_1a_1^{-1}b_1^{-1}=1 \tag{$\alpha'$} \]
\[ a_1^2b_1^{-1}a_2b_1=1 \tag{$G'$} \]
Using ($E'$) to replace $a_2b_2$ at the end of ($F'$), we get
\[ a_1a_2b_2^{-1}a_2^{-1}b_1a_1b_1^{-1}a_1^{-1}b_1^{-1}a_1^{-1}=1 \]
\[ a_2b_2^{-1}a_2^{-1}b_1a_1b_1^{-1}a_1^{-1}b_1^{-1}=1 \tag{$F'$} \]
Relations ($\alpha'$) and ($F'$) both contain the word $a_2b_2a_2^{-1}$; equating the rest of the relations ($\alpha'$) and ($F'$) gives
\[ b_1a_1^{-1}b_1^{-1}a_1b_1^{-1}=b_1a_1b_1^{-1}a_1^{-1}b_1^{-1} \]
\[ a_1^{-1}b_1^{-1}a_1=a_1b_1^{-1}a_1^{-1} \]
\[ b_1^{-1}a_1^2=a_1^2b_1^{-1} \]
so $a_1^2,b_1$ commute. Replace ($\alpha'$) with this new relation, keeping the rest in their simplified form:
\[ a_1^2b_1=b_1a_1^2 \tag{$\alpha'$} \]
\[ [b_3,a_3]b_2=b_1^{-1} \tag{$D'$} \]
\[ a_1b_1a_1a_2b_2=1 \tag{$E'$} \]
\[ a_1a_2b_2^{-1}a_2^{-1}b_1a_1^{-1}b_1^{-1}a_2b_2a_2^{-1}=1 \tag{$\gamma'$} \]
\[ a_2b_2^{-1}a_2^{-1}b_1a_1b_1^{-1}a_1^{-1}b_1^{-1}=1 \tag{$F'$} \]
\[ a_1^2b_1^{-1}a_2b_1=1 \tag{$G'$} \]
In ($G'$), we can cancel the $b_1$'s since ($\alpha'$) tells us that $a_1^2$ and $b_1$ commute. Hence ($G'$) becomes $a_2=a_1^{-2}$. Putting this in ($E'$), we see that
\[ a_1b_1a_1^{-1}b_2=1, \]
so $b_2=a_1b_1a_1^{-1}$. Now ($\alpha'$) implies that $a_2$ and $b_2$ commute, so we can cancel some terms in ($\gamma'$) and ($F'$). Our relations have now simplified to:
\[ a_1^2b_1=b_1a_1^2 \tag{$\alpha'$} \]
\[ [b_3,a_3]b_2=b_1^{-1} \tag{$D'$} \]
\[ b_2=a_1b_1^{-1}a_1^{-1} \tag{$E'$} \]
\[ a_1b_2^{-1}b_1a_1^{-1}b_1^{-1}b_2=1 \tag{$\gamma'$} \]
\[ b_2^{-1}b_1a_1b_1^{-1}a_1^{-1}b_1^{-1}=1 \tag{$F'$} \]
\[ a_2=a_1^{-2} \tag{$G'$} \]
Using ($E'$) to substitute $b_2=a_1b_1a_1^{-1}$ into ($\gamma'$), we find that
\[ a_1(a_1b_1a_1^{-1})b_1a_1^{-1}b_1^{-1}(a_1b_1^{-1}a_1^{-1})=1.\]
Cancelling the first and last $a_1$'s and using ($\alpha'$) to see that $a_1^{-1}b_1a_1=a_1b_1a_1^{-1}$, we get
\[ a_1b_1a_1^{-1}b_1a_1b_1^{-1}a_1^{-1}b_1^{-1}=1 \tag{$\gamma'$} \]
Using ($E'$) in ($F'$), we see that
\[ a_1b_1a_1^{-1}b_1a_1b_1^{-1}a_1^{-1}b_1^{-1}=1, \]
which is the same as ($\gamma'$). This lets us eliminate ($F'$) from the relations. Note that ($\gamma'$) is saying that $a_1b_1a_1^{-1}$ and $b_1$ commute. At this point, we throw in the surface relation $[a_1,b_1][a_2,b_2][a_3,b_3]=1$, which simplifies to $[a_1,b_1][a_3,b_3]=1$ since $a_2,b_2$ commute. Our relations are now
\[ a_1^2b_1=b_1a_1^2 \tag{$\alpha'$} \]
\[ [b_3,a_3]b_2=b_1^{-1} \tag{$D'$} \]
\[ b_2=a_1b_1^{-1}a_1^{-1} \tag{$E'$} \]
\[ [a_1b_1a_1^{-1},b_1]=1 \tag{$\gamma'$} \]
\[ a_2=a_1^{-2} \tag{$G'$} \]
\[ [a_1,b_1][a_3,b_3]=1 \tag{$\Sigma$} \]
Let us simplify ($D'$): using ($E'$) to substitute $b_2=a_1b_1^{-1}a_1^{-1}$ and ($\Sigma$) to substitute $[b_3,a_3]=[a_1,b_1]$ gives
\[ a_1b_1a_1^{-1}b_1^{-1}a_1b_1^{-1}a_1^{-1}=b_1^{-1}, \]
which is the same as ($\gamma'$). Hence we can eliminate ($D'$). In summary, the relations coming from the primed curves and the surface relation are:
\[ a_1^2b_1=b_1a_1^2,\; b_2=a_1b_1^{-1}a_1^{-1},\; [a_1b_1a_1^{-1},b_1]=1,\]
\[ a_2=a_1^{-2},\; [a_1,b_1][a_3,b_3]=1. \]

We now analyze the relations coming from the unprimed curves in (\ref{eq:hamadabreed}), with $\phi=t_{b_2}^k$:
\[ a_3b_3a_3^{-1}b_2b_3a_3^{-1}b_2^{-1}a_3=1 \tag{$\phi(\alpha)$} \]
\[ b_3b_2=1 \tag{$\phi(D)$} \]
\[ [b_3,a_3]b_2a_2b_2^{-k}a_3b_3a_3=1 \tag{$\phi(E)$} \]
\[ b_2^{-1}[a_3,b_3]a_3^{-1}b_2a_3=1 \tag{$\phi(\gamma)$} \]
\[ b_2^{-1}a_3[b_3,a_3]b_2a_2b_2^{-k}a_3=1 \tag{$\phi(F)$} \]
\[ a_3b_3^{-1}a_2b_2^{-k}b_2^{-1}a_3=1 \tag{$\phi(G)$} \]
Using ($\phi(D)$) to get rid of the $b_2b_3$ term in ($\phi(\alpha)$), we see that $a_3^2b_3^{-1}a_3^{-2}b_3=1$, so that $a_3^2$ and $b_3$ commute. Substituting $b_2=a_1b_1^{-1}a_1^{-1}$, ($\phi(D)$) becomes $b_3=a_1b_1a_1^{-1}$. Using ($\phi(D)$) in ($\phi(F)$) to replace $b_2^{-1}$ with $b_3$ produces ($\phi(E)$), so we can eliminate ($\phi(E)$). We use the relation ($D'$) from the primed curves to get rid of commutators in the remaining relations:
\[ a_3^2b_3=b_3a_3^2 \tag{$\phi(\alpha)$} \]
\[ b_3=a_1b_1a_1^{-1} \tag{$\phi(D)$} \]
\[ b_1a_3^{-1}b_2a_3=1 \tag{$\phi(\gamma)$} \]
\[ b_2^{-1}a_3b_1^{-1}a_2b_2^{-k}a_3=1 \tag{$\phi(F)$} \]
\[ a_3b_3^{-1}a_2b_2^{-k-1}a_3=1 \tag{$\phi(G)$} \]

The relation ($\phi(\gamma)$) implies that $b_2^{-1}a_3b_1^{-1}=a_3$, and putting this into ($\phi(F)$) gives $a_3a_2b_2^{-k}a_3=1$. Expressing $a_2$ and $b_2$ in terms of $a_1,b_1$ gives
\[ a_3a_1^{-1}b_1^ka_1^{-1}a_3=1. \tag{$\phi(F)$} \]

Substituting $b_2=a_1b_1^{-1}a_1^{-1}$ into ($\phi(\gamma)$) gives $b_1a_3^{-1}a_1b_1^{-1}a_1^{-1}a_3=1$, and this becomes $b_1a_3^{-1}b_3^{-1}a_3=1$ after using ($\phi(D)$). In other words, ($\phi(\gamma)$) has become $b_3=a_3b_1a_3^{-1}$.

Expressing $a_2$ and $b_2$ in terms of $a_1,b_1$ in ($\phi(G)$) gives
\[ a_3b_3^{-1}a_1^{-1}b_1^{k+1}a_1^{-1}a_3=1.\]
By ($\phi(F)$), we have $b_1^ka_1^{-1}a_3=a_1a_3^{-1}$, so substituting this above gives $a_3b_3^{-1}a_1^{-1}b_1a_1a_3^{-1}=1$. Cancelling the $a_3$'s gives $b_3=a_1^{-1}b_1a_1$, and since $a_1^{-1}b_1a_1=a_1b_1a_1^{-1}$ by ($\alpha'$), we get $b_3=a_1b_1a_1^{-1}$. This is precisely ($\phi(D)$), so we can get rid of ($\phi(G)$) from our relations. 

Let us write ($\phi(F)$) slightly differently: the relations $b_3=a_1b_1a_1^{-1}$ and $b_3=a_3b_1a_3^{-1}$ imply that $b_1$ commutes with $a_1^{-1}a_3$, so we can commute $b_1^k$ all the way to the right in ($\phi(F)$):
\[ a_3a_1^{-2}a_3=b_1^{-k} \tag{$\phi(F)$} \]
Summarizing, the relations from the unprimed curves collapse to:
\[ a_3^2b_3=b_3a_3^2 \tag{$\phi(\alpha)$} \]
\[ b_3=a_1b_1a_1^{-1} \tag{$\phi(D)$} \]
\[ b_3=a_3b_1a_3^{-1} \tag{$\phi(\gamma)$} \]
\[ a_3a_1^{-1}b_1^ka_1^{-1}a_3=1 \tag{$\phi(F)$} \]
Observe that the surface relation $[a_1,b_1][a_3,b_3]=1$ is equivalent to $a_1b_1a_1^{-1}b_1^{-1}=b_3a_3b_3^{-1}a_3^{-1}$, and using ($\phi(D)$) gives $b_3b_1^{-1}=b_3a_3b_3^{-1}a_3^{-1}$. Cancelling the leftmost $b_3$'s gives a relation equivalent to ($\phi(\gamma)$), so we can eliminate the surface relation.

We are done simplifying the relations; the remaining relations are ($\alpha'$), ($E'$), ($\gamma'$), ($G'$), ($\phi(\alpha)$), ($\phi(D)$), ($\phi(\gamma)$), and ($\phi(F)$). The relations ($E'$) and ($G'$) eliminate the generators $a_2,b_2$, so we are left with the following presentation of $\pi_1(X_k)$:
\[ \pi_1(X_k)\cong\langle a_1,b_1,a_3,b_3|(\alpha'),(\gamma'),(\phi(\alpha)),(\phi(D)),(\phi(\gamma)),(\phi(F))\rangle,\]
with
\[ a_1^2b_1=b_1a_1^2 \tag{$\alpha'$} \]
\[ [a_1b_1a_1^{-1},b_1]=1 \tag{$\gamma'$} \]
\[ a_3^2b_3=b_3a_3^2 \tag{$\phi(\alpha)$} \]
\[ b_3=a_1b_1a_1^{-1} \tag{$\phi(D)$} \]
\[ b_3=a_3b_1a_3^{-1} \tag{$\phi(\gamma)$} \]
\[ a_3a_1^{-2}a_3=b_1^{-k} \tag{$\phi(F)$} \]

It is not immediately clear from this presentation whether $\pi_1(X_k)$ is the of the form in Proposition \ref{prop:fundgrp}, so we make a change-of-basis. Replace the generator $a_1$ by
\[ c=a_3^{-1}a_1, \]
so that $c,b_1,a_3,b_3$ is our new generating set for $\pi_1(X_k)$.

Using the relations between $a_1,b_1,a_3$, and $b_3$, we observe the following relations between $c,b_1,a_3$, and $b_3$:
\begin{equation}
\label{eq:cb1}
\begin{split}
	cb_1c^{-1}&=a_3^{-1}a_1b_1a_1^{-1}a_3\\
	&=a_3^{-1}b_3a_3\\
	&=b_1
\end{split}
\end{equation}
\begin{equation}
\label{eq:cb3}
\begin{split}
	c^{-1}b_3c&=a_1^{-1}a_3b_3a_3^{-1}a_1\\
	&=a_1^{-1}b_1a_1\\
	&=b_3
\end{split}
\end{equation}
\[ b_1^{-k}=a_3(a_3c)^{-2}a_3=a_3c^{-1}a_3^{-1}c^{-1} \]
\[ a_3c^{-1}a_3^{-1}=b_1^{-k}c \]
\begin{equation}
\label{eq:ca3}
	a_3ca_3^{-1}=c^{-1}b_1^k
\end{equation}
\begin{equation}
\label{eq:b13}
	[b_1,b_3]=[b_1,a_1b_1a_1^{-1}]=1
\end{equation}
\begin{equation}
\label{eq:ab3}
	a_3b_3a_3^{-1}=a_3^2b_1a_3^{-2}=b_1
\end{equation}
\[ b_3=a_3b_1a_3^{-1} \tag{$\phi(\gamma)$} \]

Conversely, suppose we are only given the relations (\ref{eq:cb1}), (\ref{eq:cb3}), (\ref{eq:ca3}), (\ref{eq:b13}), (\ref{eq:ab3}), and ($\phi(\gamma)$). We will derive the relations ($\alpha'$), ($\gamma'$), ($\phi(\alpha)$), ($\phi(D)$), and ($\phi(F)$):
\begin{enumerate}
\item[$\phi(\alpha)$:] Starting with ($\phi(\gamma)$) and using (\ref{eq:ab3}), we get
\[ b_3=a_3b_1a_3^{-1}=a_3^2b_3a_3^{-2}, \]
which is ($\phi(\alpha)$).

\item[$\phi(D)$:] Using $c=a_3^{-1}a_1$ in (\ref{eq:cb1}), we get $a_3^{-1}a_1b_1=b_1a_3^{-1}a_1$, and thus $a_1b_1=a_3b_1a_3^{-1}a_1$. Using ($\phi(\gamma)$) then gives $a_1b_1=b_3a_1$, which is ($\phi(D)$).

\item[$\alpha'$:] Using $c=a_3^{-1}a_1$ in (\ref{eq:cb3}), we get $a_3^{-1}a_1b_3=b_3a_3^{-1}a_1$, and hence $a_1b_3=a_3b_3a_3^{-1}a_1$. Now use (\ref{eq:ab3}) to get $a_1b_3=b_1a_1$. Thus $b_3=a_1^{-1}b_1a_1$. We derived ($\phi(D)$) above, which says that $b_3=a_1b_1a_1^{-1}$. We see that $a_1^{-1}b_1a_1=a_1b_1a_1^{-1}$, which is equivalent to ($\alpha'$).

\item[$\gamma'$:] We derived above that $b_3=a_1b_1a_1^{-1}$. Hence the relation (\ref{eq:b13}) becomes $[b_1,a_1b_1a_1^{-1}]=1$, which is ($\gamma'$).

\item[$\phi(F)$:] In the process of deriving (\ref{eq:ca3}) from ($\phi(F)$), we only used $c=a_3^{-1}a_1$; no other relations were used. Hence (\ref{eq:ca3}) and ($\phi(F)$) are equivalent.
\end{enumerate}
Hence under the two sets of generators $a_1,b_1,a_3,b_3$ and $c=a_3^{-1}a_1,b_1,a_3,b_3$, the sets of relations ($\alpha'$)-($\phi(F)$) and (\ref{eq:cb1})-($\phi(\gamma)$) are equivalent. So
\begin{gather*}
	\pi_1(X_k)\cong \langle c,b_1,a_3,b_3|[c,b_1]=[c,b_3]=[b_1,b_3]=1,\\
	a_3b_3a_3^{-1}=b_1,a_3b_1a_3^{-1}=b_3,a_3ca_3^{-1}=c^{-1}b_1^k\rangle
\end{gather*}
From this presentation, it is clear that $\pi_1(X_k)\cong\Z^3\rtimes_A\Z$, with $\Z^3$ generated by $\{b_1,b_3,c\}$, $\Z$ generated by $a_3$, and
\[ A = \begin{pmatrix} 0 & 1 & k\\1 & 0 & 0\\ 0 & 0 & -1\end{pmatrix}. \]

If $k=2t$ is even, take
\[ P = \begin{pmatrix} 0 & 1 & 0\\0 & 0 & 1\\1 & 1 & t\end{pmatrix}, \]
so that
\[ PAP^{-1}=\begin{pmatrix} -1 & -t & 1\\0 & -1 & 0\\0 & 0 & 1\end{pmatrix}. \]
By Lemma \ref{lem:conj} and Proposotion \ref{prop:torusbndl}, $X_k\cong M(T,I;1,0)$, where
\[ T = \begin{pmatrix}-1 & -t\\0 & -1\end{pmatrix}. \]
By Proposition \ref{prop:section}, this torus bundle does not have a section.

If $k=2t+1$ is odd, take
\[ P = \begin{pmatrix} 1 & 0 & t\\1 & 1 & t\\2 & 2 & 2t+1\end{pmatrix}, \]
so that
\[ PAP^{-1} = \begin{pmatrix} -1 & -k & t+1\\0 & -1 & 1\\0 & 0 & 1\end{pmatrix}. \]
By Lemma \ref{lem:conj} and Proposition \ref{prop:torusbndl}, $X_k\cong M(T,I;-t,1)$, where
\[ T = \begin{pmatrix} -1 & -k\\0 & -1\end{pmatrix}. \]
By Proposition \ref{prop:section}, this torus bundle does not have a section.
\end{proof}
\begin{rmk}
The relations coming from the primed curves imply that $a_2$ and $b_2$ commute. Thus, given any $\begin{pmatrix} p & q\\r & s\end{pmatrix}\in\SL(2,\Z)$ we can let $\phi$ be an appropriate product of $t_{a_2}$'s and $t_{b_2}$'s so that $\phi(a_2)=a_2^pb_2^r$, $\phi(b_2)=a_2^qb_2^s$, and $\phi$ fixes the other generators. It is then not too difficult to determine the relations coming from the unprimed curves, now conjugated by $\phi$. When combined with the relations coming from the primed curves we get the following presentation for $\pi_1(X_\phi)$, where $X_\phi$ is the total space of the pencil (\ref{eq:hamadabreed}):
\begin{gather*}
	\pi_1(X_\phi)\cong\langle a_1,b_1,a_3,b_3|b_3=a_1^{2q+1}b_1^sa_1^{-1},a_3^2=a_1^{2p+1}b_1^ra_1^{-1},\\
	[a_1^2,b_1]=[a_1b_1a_1^{-1},b_1]=[a_1,b_1][a_3,b_3]=[a_3b_3a_3^{-1},b_3]=1\rangle.
\end{gather*}
It is not clear to the author how to recognize this as the fundamental group of a torus bundle.
\end{rmk}
\begin{rmk}
Taking $\phi$ to be $t_{b_3}$, $t_{b_3}t_{a_3}$, $t_{a_3}t_{b_3}t_{a_3}$, or $t_{a_3}t_{b_3}^{-1}t_{a_3}$ will all result in a fundamental group isomorphic to $\pi_1(M(-I,I;0,-1))$. It is possible, but tedious, to do these computations by hand, and we do not include them here because they are similar to our previous computations. Alternatively, one can put the original group presentations for $\pi_1(X_\phi)$ and $\pi_1(M(-I,I;0,-1))$ into a computer algebra system and see that they both simplify to the same finitely-presented group.
\end{rmk}

\section{Reverse engineering}
\label{sec:reverse}
In this section, we use the \textit{reverse engineering} technique of \cite{bk} to derive a positive factorization of the boundary twist $t_\delta$ in $\Mod(\Sigma_3^1)$ of length 12. Let $c_1,\dots,c_6$ be the chain of curves in $\Sigma_3^1$ as in Figure \ref{fig:six_chain}. 
\begin{figure}[h]
\centering{
\resizebox{90mm}{!}{
\begingroup%
  \makeatletter%
  \providecommand\color[2][]{%
    \errmessage{(Inkscape) Color is used for the text in Inkscape, but the package 'color.sty' is not loaded}%
    \renewcommand\color[2][]{}%
  }%
  \providecommand\transparent[1]{%
    \errmessage{(Inkscape) Transparency is used (non-zero) for the text in Inkscape, but the package 'transparent.sty' is not loaded}%
    \renewcommand\transparent[1]{}%
  }%
  \providecommand\rotatebox[2]{#2}%
  \newcommand*\fsize{\dimexpr\f@size pt\relax}%
  \newcommand*\lineheight[1]{\fontsize{\fsize}{#1\fsize}\selectfont}%
  \ifx\svgwidth\undefined%
    \setlength{\unitlength}{348.77807137bp}%
    \ifx\svgscale\undefined%
      \relax%
    \else%
      \setlength{\unitlength}{\unitlength * \real{\svgscale}}%
    \fi%
  \else%
    \setlength{\unitlength}{\svgwidth}%
  \fi%
  \global\let\svgwidth\undefined%
  \global\let\svgscale\undefined%
  \makeatother%
  \begin{picture}(1,0.51027636)%
    \lineheight{1}%
    \setlength\tabcolsep{0pt}%
    \put(0,0){\includegraphics[width=\unitlength,page=1]{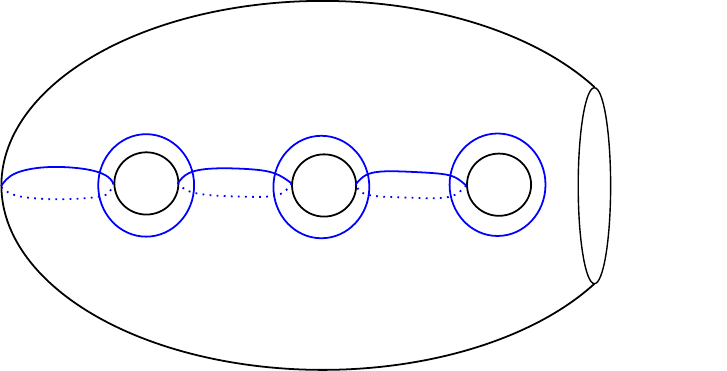}}%
    \put(0.05382552,0.28826941){\color[rgb]{0,0,0}\makebox(0,0)[lt]{\lineheight{1.25}\smash{\begin{tabular}[t]{l}$c_1$\end{tabular}}}}%
    \put(0.18075475,0.15454184){\color[rgb]{0,0,0}\makebox(0,0)[lt]{\lineheight{1.25}\smash{\begin{tabular}[t]{l}$c_2$\end{tabular}}}}%
    \put(0.30164037,0.28867776){\color[rgb]{0,0,0}\makebox(0,0)[lt]{\lineheight{1.25}\smash{\begin{tabular}[t]{l}$c_3$\end{tabular}}}}%
    \put(0.43368446,0.14809075){\color[rgb]{0,0,0}\makebox(0,0)[lt]{\lineheight{1.25}\smash{\begin{tabular}[t]{l}$c_4$\end{tabular}}}}%
    \put(0.55038649,0.28606218){\color[rgb]{0,0,0}\makebox(0,0)[lt]{\lineheight{1.25}\smash{\begin{tabular}[t]{l}$c_5$\end{tabular}}}}%
    \put(0.68702082,0.14890405){\color[rgb]{0,0,0}\makebox(0,0)[lt]{\lineheight{1.25}\smash{\begin{tabular}[t]{l}$c_6$\end{tabular}}}}%
    \put(0.83243805,0.38543613){\color[rgb]{0,0,0}\makebox(0,0)[lt]{\lineheight{1.25}\smash{\begin{tabular}[t]{l}$\delta$\end{tabular}}}}%
  \end{picture}%
\endgroup%
}
\caption{The chain of curves $c_1,\dots,c_6$ on $\Sigma_3^1$.}
\label{fig:six_chain}
}
\end{figure}

We will denote the Dehn twist $t_{\gamma}$ by the curve $\gamma$ itself, and we will further use the shorthand $i$ to denote the Dehn twist $t_{c_i}$. For example, $14\sigma2$ means $t_{c_1}t_{c_4}t_\sigma t_{c_2}$. Observe that for $i=2,\dots,6$, the braid relation and commutativity imply that
\[ i(123456)=(123456)(i-1). \]
Hence
\begin{equation*}
\begin{split}
	(123456)^6&=(123456)(123456)^5\\
	&=(12345)(123456)^51\\
	&=(12345)(123456)(123456)^41\\
	&=(12345)(12345)(123456)^421\\
	&=(12345)^2(123456)(123456)^321\\
	&=(12345)^2(12345)(123456)^3321\\
	&=(12345)^3(123456)(123456)^2321\\
	&=(12345)^3(12345)(123456)^24321\\
	&=(12345)^4(123456)(123456)4321\\
	&=(12345)^4(12345)(123456)54321\\
	&=(12345)^6(654321),
\end{split}
\end{equation*}
and it follows that
\[ (123456)^7=(12345)^6(654321123456).\]
Repeating the above argument for $(12345)^6$ shows that
\[ (12345)^6=(1234)^5(5432112345).\]
Thus
\[ (123456)^7=(1234)^5(5432112345)(654321123456).\]

Observe that by the braid relation and commutativity, the mapping class $654321123456$ fixes each of $c_1,\dots,c_5$. Thus $654321123456$ commutes with $5432112345$ and with $(1234)^5$. Similarly, the mapping class $5432112345$ commutes with each of $c_1,\dots,c_4$ and hence $5432112345$ commutes with $(1234)^5$.

Our starting point is the 6-chain relation:
\begin{equation*}
\begin{split}
	\delta&=(123456)^{14}\\
	&=(1234)^{10}(543211234{\color{red}5})({\color{red}65}4321123456)(6543211234{\color{red}56})({\color{red}5}432112345)\\
	&=\sigma (543211234)65{\color{red}6}(4321123456)(6543211234){\color{red}6}56(432112345)\\
	&=\sigma(54321123{\color{red}4})({\color{red}654}3211234){\color{red}656656}(432112345{\color{red}6})(432112345)\\
	&=\sigma(54321123)6545(3211234)(56)^3(432112345)(43211234){\color{red}65}\\
	&=\sigma_1({\color{red}655})(4321123)6545(3211234)(56)^3(432112345)(43211234)\\
	&=\sigma_1x_1(321123){\color{red}65565}45(3211234)(56)^3(432112345)(43211234)\\
	&=\sigma_1x_1(321123)x_2({\color{red}655655})(3211234)(56)^3(432112345)(43211234)\\
	&=\sigma_1x_1(321123)x_2(321123)x_3{\color{red}(655655)(56)^3}(432112345)(43211234)\\
	&=\sigma_1x_1(321123)x_2(321123)x_3{\color{red}(56)^6}(432112345)(43211234)\\
	&=\sigma_1x_1(321123)x_2(321123)x_3b(432112345)(43211234).
\end{split}
\end{equation*}
Here we used the 4-chain relation $(1234)^{10}=\sigma$, the 2-chain relation $(56)^6=b$, and the curves resulting from Hurwitz moves are
\[ \sigma_1=t_{c_6}t_{c_5}(\sigma),\;x_1=t_{c_6}t_{c_5}^2(c_4),\;x_2=t_{c_6}t_{c_5}^2t_{c_6}t_{c_5}(c_4),\;x_3=(t_{c_6}t_{c_5}^2)^2(c_4).\]

Continuing from where we left off, we keep using Hurwitz moves to produce 2-chains:
\begin{equation*}
\begin{split}
	\delta&=\sigma_1x_1({\color{red}3}21123)x_2(32112{\color{red}3})x_3b(432112345)(43211234)\\
	&=\sigma_1x_1'(21123)x_2(32112)x_3'b\cdot{\color{red}343}2112345\cdot 432112{\color{red}343}\\
	&=\sigma_1x_1'(21123)x_2(32112)x_3'b\cdot 43{\color{red}4}2112345\cdot 432112{\color{red}4}34\\
	&=\sigma_1x_1'(21123)x_2(3211{\color{red}2})x_3'b(432112)4345434(21123{\color{red}4})\\
	&={\color{red}4}\sigma_1x_1'(21123)x_2(3211)x_3''b({\color{red}2432}112)4345434(21123)\\
	&=\sigma_1'x_1''(2112){\color{red}43}x_2(3211)x_3''b(4323112)4345434(21123)\\
	&=\sigma_1'x_1''(2112)x_2'{\color{red}433}211x_3''b(4323112)4345434(21123)\\
	&=\sigma_1'x_1''(2112)x_2'x_411x_3'''\beta(4334323112)4345434({\color{red}2112}3)\\
	&=({\color{red}2112})\sigma_1'x_1''(2112)x_2'x_411x_3'''\beta({\color{red}43343}23112)4345434x_5\\
	&=\sigma_1'x_1'''(2112)(2112)x_2'x_411x_3'''\beta x_6({\color{red}433433}112)4345{\color{red}434}x_5\\
	&=\sigma_1'x_1'''(2112)^2x_2'x_411x_3'''\beta x_6{\color{red}11}x_8{\color{red}(433433)(434434)}x_7x_5\\
	&=\sigma_1'x_1'''(2112)^2x_2'x_4{\color{red}1111}y_3\beta x_6'x_8{\color{red}(34)^6}x_7x_5\\
	&=\sigma_1'x_1'''{\color{red}(2112)^21111}x_2'x_4'y_3\beta x_6'x_8dx_7x_5\\
	&=\sigma_1'x_1'''{\color{red}(12)^6}x_2'x_4'y_3\beta x_6'x_8dx_7x_5\\
	&=\sigma_1'x_1'''ex_2'x_4'y_3\beta x_6'x_8dx_7x_5.
\end{split}
\end{equation*}
The curves resulting from Hurwitz moves are:
\begin{gather*}
	\sigma_1'=t_{c_4}(\sigma_1),\; \beta=t_{c_4}t_{c_3}^2(b),\\
	x_1'=t_{c_3}^{-1}(x_1),\; x_1''=t_{c_4}(x_1'),\; x_1'''=t_{c_2}t_{c_1}^2t_{c_2}(x_1''),\\
	x_2'=t_{c_4}t_{c_3}(x_2),\\
	x_3'=t_{c_3}(x_3),\; x_3''=t_{c_2}(x_3'),\; x_3'''=t_{c_4}t_{c_3}^2(x_3''),\; y_3=t_{c_1}^{-2}(x_3'''),\\
	x_4=t_{c_4}t_{c_3}^2(c_2),\; x_4'=t_{c_1}^{-4}(x_4),\\
	x_5=t_{c_2}t_{c_1}^2t_{c_2}(c_3),\\
	x_6=t_{c_4}t_{c_3}^2t_{c_4}t_{c_3}(c_2),\; x_6'=t_{c_1}^{-2}(x_6),\\
	x_7=(t_{c_4}t_{c_3}t_{c_4})^{-1}(c_5),\\
	x_8=(t_{c_4}t_{c_3}^2)^2(c_2),
\end{gather*}
and we used the 2-chain relations $(34)^6=d$ and $(12)^6=e$.

\begin{figure}[h]
\centering{
\resizebox{90mm}{!}{
\begingroup%
  \makeatletter%
  \providecommand\color[2][]{%
    \errmessage{(Inkscape) Color is used for the text in Inkscape, but the package 'color.sty' is not loaded}%
    \renewcommand\color[2][]{}%
  }%
  \providecommand\transparent[1]{%
    \errmessage{(Inkscape) Transparency is used (non-zero) for the text in Inkscape, but the package 'transparent.sty' is not loaded}%
    \renewcommand\transparent[1]{}%
  }%
  \providecommand\rotatebox[2]{#2}%
  \newcommand*\fsize{\dimexpr\f@size pt\relax}%
  \newcommand*\lineheight[1]{\fontsize{\fsize}{#1\fsize}\selectfont}%
  \ifx\svgwidth\undefined%
    \setlength{\unitlength}{293.19696141bp}%
    \ifx\svgscale\undefined%
      \relax%
    \else%
      \setlength{\unitlength}{\unitlength * \real{\svgscale}}%
    \fi%
  \else%
    \setlength{\unitlength}{\svgwidth}%
  \fi%
  \global\let\svgwidth\undefined%
  \global\let\svgscale\undefined%
  \makeatother%
  \begin{picture}(1,0.65166912)%
    \lineheight{1}%
    \setlength\tabcolsep{0pt}%
    \put(0,0){\includegraphics[width=\unitlength,page=1]{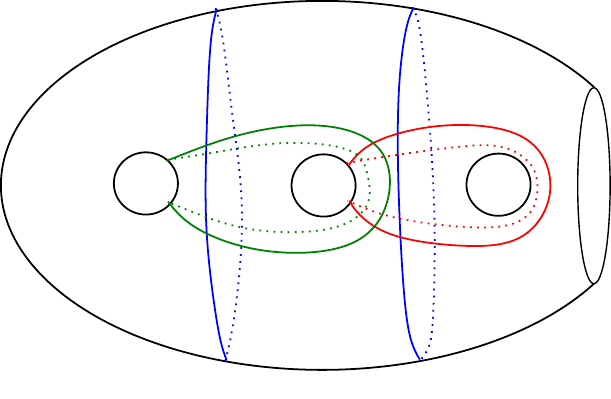}}%
    \put(0.6795438,0.01290877){\color[rgb]{0,0,0}\makebox(0,0)[lt]{\lineheight{1.25}\smash{\begin{tabular}[t]{l}$\sigma$\end{tabular}}}}%
    \put(0.79969314,0.20301819){\color[rgb]{0,0,0}\makebox(0,0)[lt]{\lineheight{1.25}\smash{\begin{tabular}[t]{l}$b$\end{tabular}}}}%
    \put(0.46205856,0.18780947){\color[rgb]{0,0,0}\makebox(0,0)[lt]{\lineheight{1.25}\smash{\begin{tabular}[t]{l}$d$\end{tabular}}}}%
    \put(0.34951387,0.009867){\color[rgb]{0,0,0}\makebox(0,0)[lt]{\lineheight{1.25}\smash{\begin{tabular}[t]{l}$e$\end{tabular}}}}%
  \end{picture}%
\endgroup%
}
\caption{The curves $\sigma,b,d,e$ resulting from chain substitutions.}
\label{fig:chain_subs}
}
\end{figure}

In summary, we have derived the following relation
\begin{equation}
\label{eq:reveng}
	\sigma_1'x_1'''ex_2'x_4'y_3\beta x_6'x_8dx_7x_5=\delta.
\end{equation}
This relation is of type $(n,s)=(8,4)$ (8 nonseparating and 4 separating curves). These curves are shown on $\Sigma_3$ (with $\delta$ capped off) in Figure \ref{fig:reveng1}. Since these curves are fixed setwise by the hyperelliptic involution $\iota$, in Figure \ref{fig:reveng1quot} we also show their images in $S_{0,8}$, the quotient of $\Sigma_3$ by $\iota$.

\begin{figure}[h]
\centering{
\resizebox{140mm}{!}{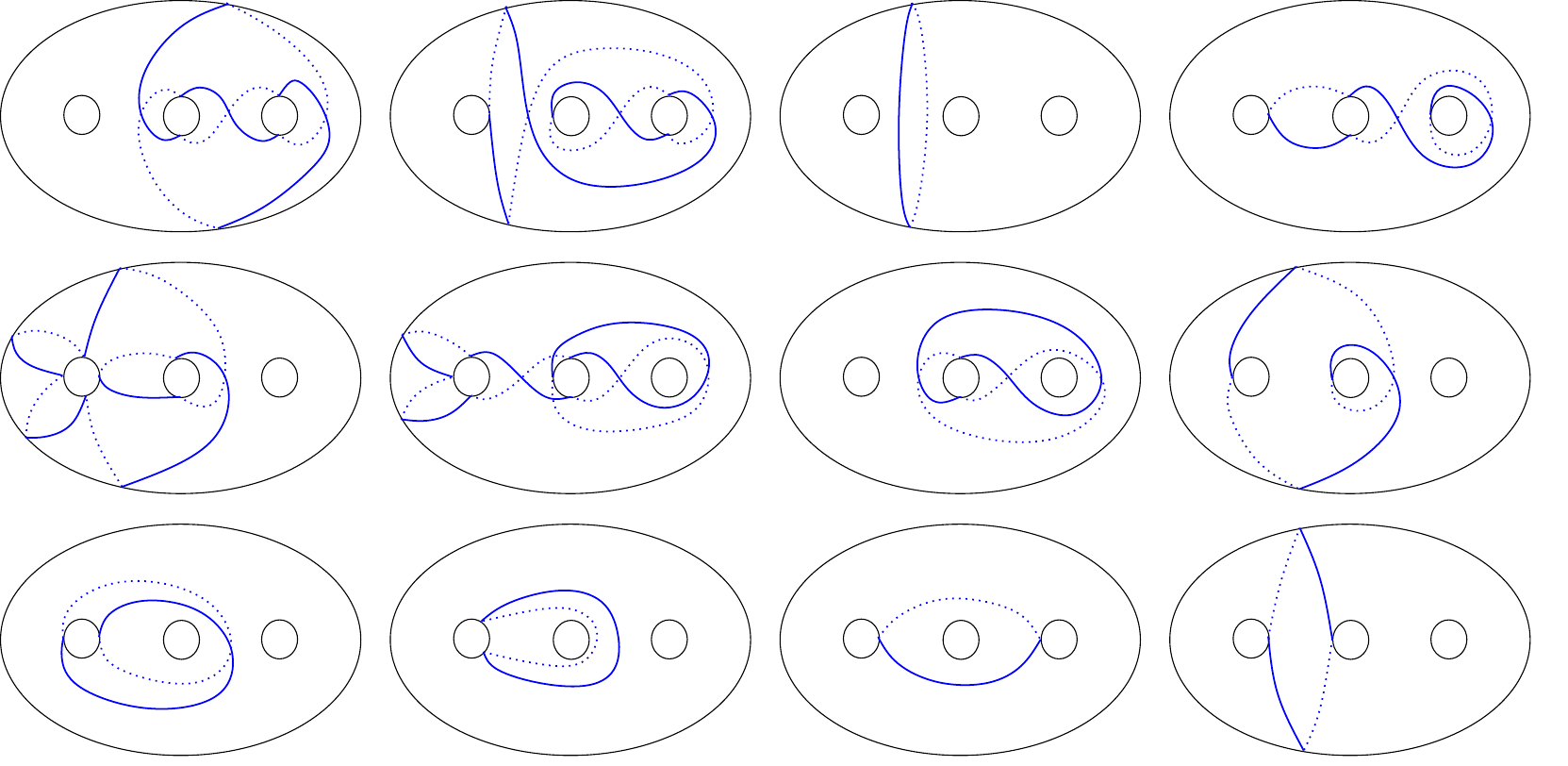}
\caption{The curves in the relation (\ref{eq:reveng}) shown on $\Sigma_3$.}
\label{fig:reveng1}
}
\end{figure}
\begin{figure}[h]
\centering{
\resizebox{140mm}{!}{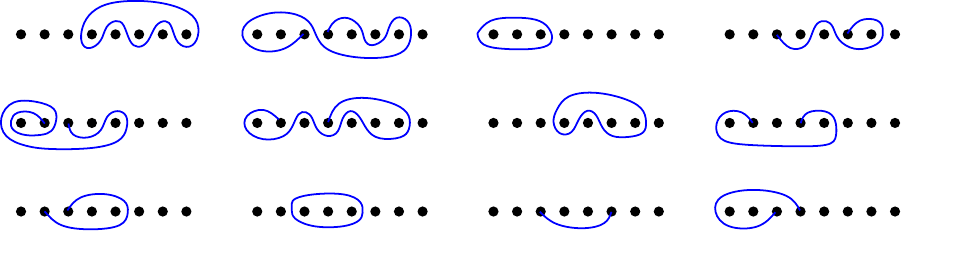}
\caption{The curves in the relation (\ref{eq:reveng}) shown on $S_{0,8}$.}
\label{fig:reveng1quot}
}
\end{figure}

In terms of the standard generators of $\pi_1(\Sigma_3)$, these curves can be written as:
\begin{gather*}
	\sigma_1'=[b_3^{-1}a_3^{-1}b_2a_2a_3,b_3a_3]\\
	x_1'''=b_1[a_1,b_1]a_2a_3b_3a_2^{-1}b_2^{-1}a_3b_3\\
	e=[a_1,b_1]\\
	x_2'=b_2^{-1}a_2^{-1}b_1b_2^{-1}a_3b_3a_3\\
	x_4'=a_1b_1^3a_2b_2b_1^{-1}a_2b_2\\
	y_3=b_2^{-1}a_2^{-1}b_1a_1b_1^2b_2^{-1}a_3[b_3,a_3]b_2a_2a_3\\
	\beta=a_3^{-1}a_2^{-1}b_2^{-1}a_3[b_3,a_3]b_2a_2\\
	x_6'=a_1b_1a_2b_2a_2\\
	x_8=a_1b_1^{-1}a_2[b_2,a_2]b_1a_2\\
	d=a_2^{-1}b_1^{-1}a_2[b_2,a_2]b_1\\
	x_7=b_1^{-1}a_2a_3b_3a_3^{-1}\\
	x_5=b_1[a_1,b_1]a_2b_2a_2^{-1}
\end{gather*}

The positive factorization (\ref{eq:reveng}) specifies a Lefschetz fibration with total space $Z$. Note that $e(Z)=2(-4)+12=4$ and from the signature formula for hyperelliptic genus-3 Lefschetz fibrations
\[ \sigma(Z)=-\frac{1}{7}(4n-s)\]
we find that $\sigma(Z)=-4$. As for the fundamental group, since this fibration has a section, $\pi_1(Z)$ is obtained from $\pi_1(\Sigma_3)$ by dividing out the vanishing cycles.

We start by simplifying the relations $x_5$ and $x_6'$. The relation $e$ says that $[a_1,b_1]=1$, and hence $x_5$ becomes $b_1=a_2b_2^{-1}a_2^{-1}$. Putting this in $x_6'$ gives $a_1=a_2^{-2}$. Using these expressions for $a_1,b_1$ in $x_4'$ we see that
\begin{equation*}
\begin{split}
	1&=a_2^{-2}(a_2b_2^{-3}a_2^{-1})a_2b_2(a_2b_2a_2^{-1})a_2b_2\\
	&=a_2^{-1}b_2^{-3}b_2a_2b_2^2\\
	&=a_2^{-1}b_2^{-2}a_2b_2^2.
\end{split}
\end{equation*}
Notice that $\beta$ can be rewritten as $b_2^{-1}a_3[b_3,a_3]b_2a_2=a_2a_3$. Putting this into $y_3$ and using the fact that $[a_1,b_1]=1$, we may simplify $y_3$ to
\[ b_2^{-1}a_2^{-1}a_1b_1^3a_2a_3^2=1.\]
Writing $a_1,b_1$ in terms of $a_2,b_2$ using $x_5,x_6'$, we see that
\begin{equation*}
\begin{split}
	1&=b_2^{-1}a_2^{-3}(a_2b_2^{-3}a_2^{-1})a_2a_3^2\\
	&=b_2^{-1}a_2^{-2}b_2^{-3}a_3^2.
\end{split}
\end{equation*}

Our relations have now been simplified to
\[ [b_3^{-1}a_3^{-1}b_2a_2a_3,b_3a_3]=1 \tag{$\sigma_1'$}\]
\[ b_1a_2a_3b_3a_2^{-1}b_2^{-1}a_3b_3=1 \tag{$x_1'''$}\]
\[ [a_1,b_1]=1 \tag{$e$}\]
\[ b_2^{-1}a_2^{-1}b_1b_2^{-1}a_3b_3a_3=1 \tag{$x_2'$}\]
\[ [a_2,b_2^2]=1 \tag{$x_4'$} \]
\[ b_2^{-1}a_2^{-2}b_2^{-3}a_3^2=1 \tag{$y_3$}\]
\[ a_3^{-1}a_2^{-1}b_2^{-1}a_3[b_3,a_3]b_2a_2=1 \tag{$\beta$}\]
\[ a_1=a_2^{-2} \tag{$x_6'$}\]
\[ a_1b_1^{-1}a_2[b_2,a_2]b_1a_2=1 \tag{$x_8$}\]
\[ a_2^{-1}b_1^{-1}a_2[b_2,a_2]b_1=1 \tag{$d$}\]
\[ b_1^{-1}a_2a_3b_3a_3^{-1}=1 \tag{$x_7$}\]
\[ b_1=a_2b_2^{-1}a_2^{-1} \tag{$x_5$}\]
\[ [a_2,b_2][a_3,b_3]=1 \tag{$\Sigma$}\]

Substituting ($x_5'$) and ($x_6$) into ($e$) gives:
\begin{gather*}
	a_1b_1=b_1a_1\\
	a_2^{-2}a_2b_2^{-1}a_2^{-1}=a_2b_2^{-1}a_2^{-1}a_2^{-2}\\
	a_2^{-1}b_2^{-1}a_2^{-1}=a_2b_2^{-1}a_2^{-3}\\
	b_2^{-1}a_2^2=a_2^2b_2^{-1}\\
	[b_2,a_2^2]=1.
\end{gather*}
Using this in ($y_3$), we get $a_2^{-2}b_2^{-4}a_3^2=1$.

Substituting ($x_5'$) into ($e$), ($x_7$), and ($x_1'''$) gives
\begin{gather*}
	a_2b_2^{-1}a_3b_3a_2^{-1}b_2^{-1}a_3b_3=1,\\
	a_2b_2a_3b_3a_3^{-1}=1,\\
	a_2b_2^{-1}a_3b_3a_2^{-1}b_2^{-1}a_3b_3=1.
\end{gather*}
We will now eliminate the relation ($x_2'$): substituting ($x_5'$) into ($x_2'$) gives $b_2^{-2}a_2^{-1}b_2^{-1}a_3b_3a_3=1$, and using ($x_4'$) gives $a_2^{-1}b_2^{-3}a_3b_3a_3=1$, which we will rewrite as $a_3b_3a_3=b_2^3a_2$. Multiplying both sides on the right by $a_3^{-2}$ and using the relation ($y_3$) $a_2^{-2}b_2^{-4}a_3^2=1$ gives
\[ a_3b_3a_3^{-1}=b_2^3a_2a_3^{-2}=b_2^3a_2^{-1}b_2^{-4}=b_2^{-1}a_2^{-1},\]
where we used ($x_4'$) in the last step. This is the same relation as ($x_7$), so we may eliminate ($x_2'$).

We again summarize our relations so far:
\[ [b_3^{-1}a_3^{-1}b_2a_2a_3,b_3a_3]=1 \tag{$\sigma_1'$}\]
\[ a_2b_2^{-1}a_3b_3a_2^{-1}b_2^{-1}a_3b_3=1 \tag{$x_1'''$}\]
\[ [b_2,a_2^2]=1 \tag{$e$}\]
\[ [a_2,b_2^2]=1 \tag{$x_4'$} \]
\[ a_2^{-2}b_2^{-4}a_3^2=1 \tag{$y_3$}\]
\[ a_3^{-1}a_2^{-1}b_2^{-1}a_3[b_3,a_3]b_2a_2=1 \tag{$\beta$}\]
\[ a_1=a_2^{-2} \tag{$x_6'$}\]
\[ a_1b_1^{-1}a_2[b_2,a_2]b_1a_2=1 \tag{$x_8$}\]
\[ a_2^{-1}b_1^{-1}a_2[b_2,a_2]b_1=1 \tag{$d$}\]
\[ a_2b_2a_3b_3a_3^{-1}=1 \tag{$x_7$}\]
\[ b_1=a_2b_2^{-1}a_2^{-1} \tag{$x_5$}\]
\[ [a_2,b_2][a_3,b_3]=1 \tag{$\Sigma$}\]

We continue simplifying ($x_7$): multiply both sides on the right by $b_3^{-1}$ and use ($\Sigma$) to get:
\begin{gather*}
	a_2b_2a_3b_3a_3^{-1}b_3^{-1}=b_3^{-1}\\
	a_2b_2b_2a_2b_2^{-1}a_2^{-1}=b_3^{-1}\\
	a_2b_2^2a_2b_2^{-1}a_2^{-1}=b_3^{-1}.
\end{gather*}
Now use the relations ($e$) and ($x_4'$) to rewrite the left side: $b_2a_2=b_3^{-1}$.

We will now eliminate ($d$): use ($x_5$) to rewrite ($d$) as $b_2[b_2,a_2]a_2b_2^{-1}a_2^{-1}=1$, and expand to get $b_2^2a_2b_2^{-2}a_2^{-1}=1$. This is precisely ($x_4'$).

We can simplify ($\beta$) by using ($\Sigma$) to replace $[b_3,a_3]$ with $[a_2,b_2]$, and expanding to get
\[ a_3^{-1}a_2^{-1}b_2^{-1}a_3a_2b_2=1.\]
We can now eliminate ($\sigma_1'$): expand the commutator and cancel terms to get
\[ b_3^{-1}a_3^{-1}b_2a_2a_3b_3a_2^{-1}b_2^{-1}a_3b_3a_3^{-1}b_3^{-1}=1.\]
Use ($\Sigma$) to replace $a_3b_3a_3^{-1}b_3^{-1}$ with $b_2a_2b_2^{-1}a_2^{-1}$: now we have
\[ b_3^{-1}a_3^{-1}b_2a_2a_3b_3b_2^{-1}a_2^{-1}=1.\]
Using our modified ($\beta$), we can replace $b_2a_2a_3$ with $a_3a_2b_2$, and cancelling the $a_3$'s gives
\[ b_3^{-1}a_2b_2b_3b_2^{-1}a_2^{-1}=1.\]
Using our modified ($x_7$), we can replace $b_3$ with $a_2^{-1}b_2^{-1}$:
\begin{equation*}
\begin{split}
	1&=b_2a_2a_2b_2a_2^{-1}b_2^{-1}b_2^{-1}a_2^{-1}\\
	&=b_2a_2^2b_2a_2^{-1}b_2^{-2}a_2^{-1}\\
	&=b_2^2a_2b_2^{-2}a_2^{-1},
\end{split}
\end{equation*}
using ($e$) in the last step. This is the same as ($x_4'$). 

Before moving on, let us simplify ($\Sigma$): use the modified ($x_7$) to replace $b_3$ with $a_2^{-1}b_2^{-1}$ and get
\[ 1=a_2b_2a_2^{-1}b_2^{-1}a_3(a_2^{-1}b_2^{-1})a_3^{-1}(b_2a_2).\]
Move the last $a_2$ to the front and commute the $a_2^2$ past the $b_2$ using ($e$) to get
\[ 1=b_2a_2b_2^{-1}a_3a_2^{-1}b_2^{-1}a_3^{-1}b_2. \]
Now move the last $b_2$ to the front and commute the $b_2^2$ past the $a_2$ using ($x_4'$):
\[ 1=a_2b_2a_3a_2^{-1}b_2^{-1}a_3^{-1}.\]

Our relations have simplified to:
\[ a_2b_2^{-1}a_3b_3a_2^{-1}b_2^{-1}a_3b_3=1 \tag{$x_1'''$}\]
\[ [b_2,a_2^2]=1 \tag{$e$}\]
\[ [a_2,b_2^2]=1 \tag{$x_4'$} \]
\[ a_2^{-2}b_2^{-4}a_3^2=1 \tag{$y_3$}\]
\[ a_3^{-1}a_2^{-1}b_2^{-1}a_3a_2b_2=1 \tag{$\beta$}\]
\[ a_1=a_2^{-2} \tag{$x_6'$}\]
\[ a_1b_1^{-1}a_2[b_2,a_2]b_1a_2=1 \tag{$x_8$}\]
\[ b_3=a_2^{-1}b_2^{-1} \tag{$x_7$}\]
\[ b_1=a_2b_2^{-1}a_2^{-1} \tag{$x_5$}\]
\[ a_2b_2a_3a_2^{-1}b_2^{-1}a_3^{-1}=1 \tag{$\Sigma$}\]

We will now eliminate ($x_1'''$): use ($x_7$) to replace $b_3$ with $a_2^{-1}b_2^{-1}$ and get
\[ a_2b_2^{-1}a_3a_2^{-1}b_2^{-1}a_2^{-1}b_2^{-1}a_3a_2^{-1}b_2^{-1}=1,\]
move the first $a_2$ to the end and rewrite $a_2^{-1}b_2^{-1}a_2=a_2b_2^{-1}a_2^{-1}$ using ($e$):
\[ b_2^{-1}a_3a_2^{-1}b_2^{-1}a_2^{-1}b_2^{-1}a_3a_2b_2^{-1}a_2^{-1}.\]
Use ($\Sigma$) to replace $a_3a_2^{-1}b_2^{-1}$ with $b_2^{-1}a_2^{-1}a_3$:
\[ b_2^{-2}a_2^{-1}a_3a_2^{-1}b_2^{-1}a_3a_2b_2^{-1}a_2^{-1}=1.\]
Now use ($\beta$) to replace $a_2^{-1}b_2^{-1}a_3a_2$ with $a_3b_2^{-1}$:
\[  b_2^{-2}a_2^{-1}a_3a_3b_2^{-1}b_2^{-1}a_2^{-1}=1.\]
Collecting terms and commuting $a_2^{-1}$ past $b_2^{-2}$ gives $b_2^{-4}a_2^{-2}a_3^2=1$, which is ($y_3$). Hence ($x_1'''$) can be eliminated.

Let us now eliminate ($\Sigma$): starting from $a_2b_2a_3=a_3b_2a_2$, multiply both sides on the right by $a_3$:
\[ a_2b_2a_3^2=a_3b_2a_2a_3.\]
Note that ($y_3$) implies that $a_3^2=a_2^2b_2^4$, and thus $a_3^2$ commutes with $a_2$ and $b_2$ by ($x_4'$) and ($e$). Hence we can move the $a_3^2$ on the left side to the front:
\[ a_3^2a_2b_2=a_3b_2a_2a_3.\]
Cancelling $a_3$ on the left gives $a_3a_2b_2=b_2a_2a_3$, which is ($\beta$). Hence ($\Sigma$) can be eliminated.

We can also eliminate ($x_8$): use ($x_6'$) and ($x_5$) to express $a_1,b_1$ in terms of $a_2,b_2$:
\[ a_2^{-2}(a_2b_2a_2^{-1})a_2(b_2a_2b_2^{-1}a_2^{-1})(a_2b_2^{-1}a_2^{-1})a_2=1.\]
Cancelling $a_2$'s and collecting terms gives $a_2^{-1}b_2^2a_2b_2^{-2}=1$, which is ($x_4'$). To summarize, our relations are now:
\[ [b_2,a_2^2]=1 \tag{$e$}\]
\[ [a_2,b_2^2]=1 \tag{$x_4'$} \]
\[ a_2^{-2}b_2^{-4}a_3^2=1 \tag{$y_3$}\]
\[ a_3^{-1}a_2^{-1}b_2^{-1}a_3a_2b_2=1 \tag{$\beta$}\]
\[ a_1=a_2^{-2} \tag{$x_6'$}\]
\[ b_3=a_2^{-1}b_2^{-1} \tag{$x_7$}\]
\[ b_1=a_2b_2^{-1}a_2^{-1} \tag{$x_5$}\]
The relations ($x_6'$), ($x_7$), and ($x_5$) eliminate the generators $a_1,b_1,a_3$, so that
\[ \pi_1(Z)\cong\langle a_2,b_2,a_3|[a_2,b_2^2]=[b_2,a_2^2]=1,\; a_2^{-2}b_2^{-4}a_3^2=1,\; a_3a_2b_2=b_2a_2a_3 \rangle. \]

We are getting close to recognizing $\pi_1(Z)$ as the fundamental group of a torus bundle over a torus. Replace the generator $a_3$ with the generator $t=b_2^{-2}a_2^{-1}a_3$ or equivalently, $a_3=a_2b_2^2t$. The third relation $a_2^{-2}b_2^{-4}a_3^2=1$ becomes
\begin{equation}
\label{eq:reln3}
\begin{split}
	1&=a_2^{-2}b_2^{-4}a_2b_2^2ta_2b_2^2t\\
	&=b_2^{-2}a_2^{-1}ta_2b_2^2t,
\end{split}
\end{equation}
and the last relation $a_3a_2b_2=b_2a_2a_3$ becomes
\begin{equation}
\label{eq:reln4}
\begin{split}
	a_2b_2^2ta_2b_2&=b_2a_2a_2b_2^2t\\
	&=a_2^2b_2^3t\\
	a_2ta_2b_2&=a_2^2b_2t\\
	a_2^{-1}ta_2&=b_2tb_2^{-1}.
\end{split}
\end{equation}
Use (\ref{eq:reln4}) to rewrite (\ref{eq:reln3}):
\[ 1=b_2^{-2}(b_2tb_2^{-1})b_2^2t=b_2^{-1}tb_2t.\]
Moving the $b_2^{-1}$ from the front to the back gives $1=tb_2tb_2^{-1}$, and hence $b_2tb_2^{-1}=t^{-1}$. Thus (\ref{eq:reln4}) can be written as $a_2^{-1}ta_2=t^{-1}$. Now we have
\[ \pi_1(Z)\cong\langle a_2,b_2,t|[a_2,b_2^2]=[b_2,a_2^2]=1,\;a_2^{-1}ta_2=t^{-1},\; b_2^{-1}tb_2=t^{-1}\rangle. \]

Hang tight, we're almost there. Let $M=M(-I,-I;1,0)$. Observe that by Proposition \ref{prop:fundgrp}, $\pi_1(M)$ is generated by $\alpha,\beta,\sigma,\tau$ with relations
\begin{gather*}
	 \sigma\tau=\tau\sigma,\; \alpha\beta\alpha^{-1}\beta^{-1}=\sigma,\; \alpha^{-1}\sigma\alpha=\sigma^{-1},\\
	 \beta^{-1}\sigma\beta=\sigma^{-1},\alpha^{-1}\tau\alpha=\tau^{-1},\beta^{-1}\tau\beta=\tau^{-1}.
\end{gather*}
The last two relations imply that $\alpha^{\epsilon}\tau^{\epsilon'}=\tau^{-\epsilon'}\alpha^{\epsilon}$ and $\beta^{\epsilon}\tau^{\epsilon'}=\tau^{-\epsilon'}\beta^{\epsilon}$, where $\epsilon,\epsilon'\in\{-1,1\}$. Using the second relation in the first $\sigma\tau=\tau\sigma$ and our observations, we see that
\begin{equation*}
\begin{split}
	\tau\alpha\beta\alpha^{-1}\beta^{-1}&=\alpha\beta\alpha^{-1}\beta^{-1}\tau\\
	&=\alpha\beta\alpha^{-1}\tau^{-1}\beta^{-1}\\
	&=\alpha\beta\tau\alpha^{-1}\beta^{-1}\\
	&=\alpha\tau^{-1}\beta\alpha^{-1}\beta^{-1}\\
	&=\tau\alpha\beta\alpha^{-1}\beta^{-1},
\end{split}
\end{equation*}
so the relation $\sigma\tau=\tau\sigma$ can be eliminated. Using the second relation in the third gives
\begin{gather*}
	\alpha^{-1}(\alpha\beta\alpha^{-1}\beta^{-1})\alpha=\beta\alpha\beta^{-1}\alpha^{-1}\\
	\alpha^{-1}\beta^{-1}\alpha=\alpha\beta^{-1}\alpha^{-1}\\
	\beta^{-1}\alpha^2=\alpha^2\beta^{-1},
\end{gather*}
and using the second relation in the fourth gives
\begin{gather*}
	\beta^{-1}(\alpha\beta\alpha^{-1}\beta^{-1})\beta=\beta\alpha\beta^{-1}\alpha^{-1}\\
	\beta^{-1}\alpha\beta=\beta\alpha\beta^{-1}\\
	\alpha\beta^2=\beta^2\alpha.
\end{gather*}
Now the second relation can be used to eliminate the generator $\sigma$, so
\[ \pi_1(M)\cong\langle\alpha,\beta,\tau|[\alpha,\beta^2]=[\beta,\alpha^2]=1,\;\alpha^{-1}\tau\alpha=\tau^{-1},\;\beta^{-1}\tau\beta=\tau^{-1}\rangle.\]
It is now clear that $\pi_1(Z)\cong\pi_1(M)$ under the isomorphism sending $a_2\mapsto\alpha,b_2\mapsto\beta,t\mapsto\tau$.

\begin{rmk}
The torus bundles in our previous sections were all of the form $M(T,I;m,n)$. The above torus bundle $M(-I,-I;1,0)$ is also of this form. To see this, use (1) of Proposition 2 in \cite{sf} with $A=B=-I$,
\[ P=\begin{pmatrix}1 & 1\\0 & 1\end{pmatrix}, \]
and $(m,n)=(1,0)$ to get that $M(-I,-I;1,0)\cong M(-I,I;1,0)$ as torus bundles.
\end{rmk}

For reasons that will become clear in the next section, we will modify the relation (\ref{eq:reveng}) by a global conjugation $\phi$. We describe $\phi$ by its action on $S_{0,8}$ in Figure \ref{fig:genus3diffeo}, and images of the vanishing cycles in (\ref{eq:reveng}) under $\phi$ are shown in Figure \ref{fig:reveng2quotient}. Going forward, we will always work with these modified curves, and hence we will not bother renaming the curves. Hopefully this does not cause any confusion.

\begin{figure}[h]
\centering{
\resizebox{50mm}{!}{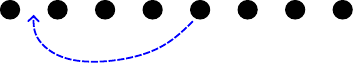}
\caption{A diffeomorphism $\phi$ of $\Sigma_3$, given by its action on $S_{0.8}$.}
\label{fig:genus3diffeo}
}
\end{figure}
\begin{figure}[h]
\centering{
\resizebox{140mm}{!}{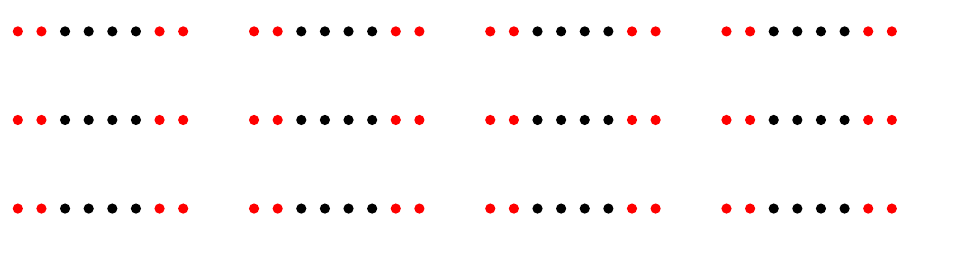}
\caption{The curves in (\ref{eq:reveng}) in $S_{0,8}$, after applying the diffeomorphism $\phi$.}
\label{fig:reveng2quotient}
}
\end{figure}

\section{Reverse breeding}
In this section we will take the reverse-engineered relation (\ref{eq:reveng}) and recognize it as the result of breeding two genus-2 relations of type $(4,3)$, as was done in Section \ref{sec:breed43}. To start, note that there are 4 fixed points of $\iota$ on $\Sigma_3$ that are disjoint from all 12 vanishing cycles. These points are marked in red in Figure \ref{fig:reveng2quotient}. Since our fibration is hyperelliptic, we have located 4 $(-1)$-sections, following an argument in Section 4.2 of \cite{baykur}. This gives a lift of our factorization from $\Mod(\Sigma_3)$ to $\Mod(\Sigma_3^4)$:
\begin{equation}
\label{eq:reveng_pencil}
	\sigma_1'x_1'''ex_2'x_4'y_3\beta x_6'x_8dx_7x_5=\delta_1\delta_2\delta_3\delta_4.
\end{equation}

\begin{figure}[h]
\centering{
\resizebox{140mm}{!}{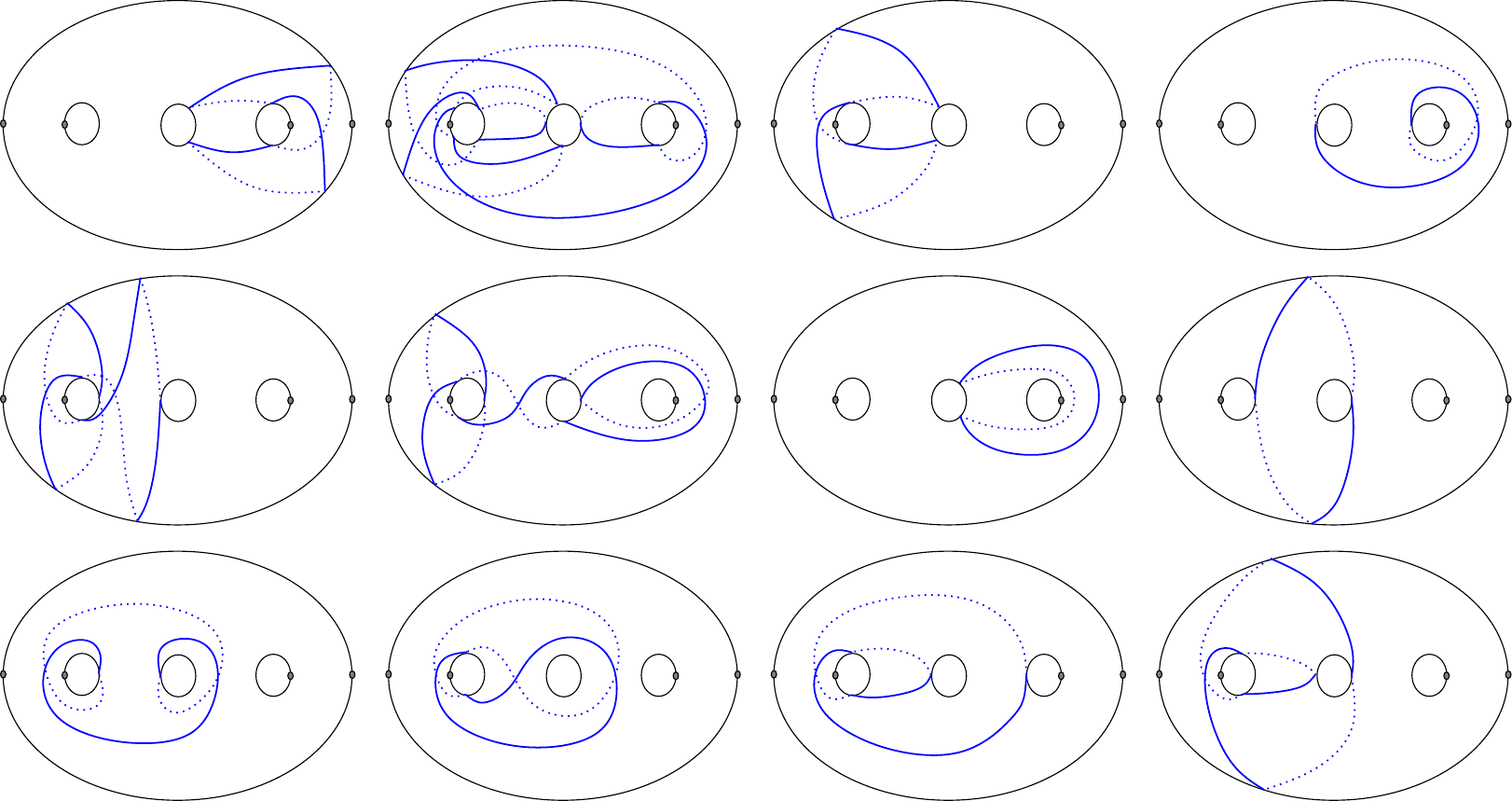}
\caption{The curves in (\ref{eq:reveng_pencil}).}
\label{fig:reveng2}
}
\end{figure}

Note that this specifies the total space of a genus-3 Lefschetz pencil $X$ with four basepoints. Since $X$ is obtained from $Z$ by blowing down four $(-1)$-sections, we have $e(X)=\sigma(X)=0$. Hence $X$ is a SCY, and by our computation of $\pi_1(X)\cong\pi_1(Z)$ in the previous section, $X$ is homeomorphic to $M(-I,-I;1,0)\cong M(-I,I;1,0)$.

Let us first do some Hurwitz moves to (\ref{eq:reveng_pencil}) to group the terms in red:
\begin{equation*}
\begin{split}
	\delta_1\delta_2\delta_3\delta_4&=\sigma_1'x_1'''{\color{red}e}x_2'{\color{red}x_4'}y_3\beta {\color{red}x_6'x_8d}x_7{\color{red}x_5}\\
	&={\color{red}e}\sigma_1'y_1x_2''{\color{red}x_4'}y_3\beta{\color{red}x_6'x_8dx_5}x_7'\\
	&=\sigma_1'y_1x_2''{\color{red}x_4'}\beta y_3'{\color{red}x_6'x_8dx_5}x_7'{\color{red}e}\\
	&=\sigma_1'y_1x_2''\beta {\color{red}x_4'} y_3'{\color{red}x_6'x_8dx_5}{\color{red}e}x_7'\\
	&=\sigma_1'y_1x_2''\beta y_3''{\color{red}x_4'}{\color{red}x_6'x_8dx_5}{\color{red}e}x_7'\\
	&=x_7'\sigma_1'y_1x_2''\beta y_3''{\color{red}x_4'}{\color{red}x_6'x_8dx_5}{\color{red}e}.
\end{split}
\end{equation*}
In the first step, we moved $e$ to the left of $x_1'''$ then commuted $e$ past $\sigma_1'$, and also moved $x_5$ to the left of $x_7$. Keeping track of the Hurwitz moves, we have
\[ y_1=t_e^{-1}(x_1'''),\;x_7'=t_{x_5}^{-1}(x_7).\]
In the second step we shuffled $e$ from the front to the back and moved $\beta$ to the left of $y_3$, giving $y_3'=t_\beta^{-1}(y_3)$. In the third step we commuted $\beta$ and $x_4'$ and commuted $e$ and $x_7'$. In the fourth step we moved $x_4'$ to the right of $y_3'$, obtaining
\[ y_3''=t_{x_4'}(y_3').\]
Lastly, we shuffled $x_7'$ from the back to the front. The resulting factorization
\begin{equation}
\label{eq:revbreed}
	x_7'\sigma_1'y_1x_2''\beta y_3''{\color{red}x_4'}{\color{red}x_6'x_8dx_5}{\color{red}e}=\delta_1\delta_2\delta_3\delta_4
\end{equation}
is Hurwitz equivalent to (\ref{eq:reveng_pencil}).

\begin{figure}[h]
\centering{
\resizebox{90mm}{!}{
\begingroup%
  \makeatletter%
  \providecommand\color[2][]{%
    \errmessage{(Inkscape) Color is used for the text in Inkscape, but the package 'color.sty' is not loaded}%
    \renewcommand\color[2][]{}%
  }%
  \providecommand\transparent[1]{%
    \errmessage{(Inkscape) Transparency is used (non-zero) for the text in Inkscape, but the package 'transparent.sty' is not loaded}%
    \renewcommand\transparent[1]{}%
  }%
  \providecommand\rotatebox[2]{#2}%
  \newcommand*\fsize{\dimexpr\f@size pt\relax}%
  \newcommand*\lineheight[1]{\fontsize{\fsize}{#1\fsize}\selectfont}%
  \ifx\svgwidth\undefined%
    \setlength{\unitlength}{335.35952591bp}%
    \ifx\svgscale\undefined%
      \relax%
    \else%
      \setlength{\unitlength}{\unitlength * \real{\svgscale}}%
    \fi%
  \else%
    \setlength{\unitlength}{\svgwidth}%
  \fi%
  \global\let\svgwidth\undefined%
  \global\let\svgscale\undefined%
  \makeatother%
  \begin{picture}(1,0.12533504)%
    \lineheight{1}%
    \setlength\tabcolsep{0pt}%
    \put(0,0){\includegraphics[width=\unitlength,page=1]{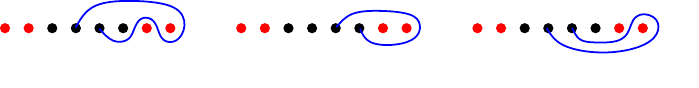}}%
    \put(0.14095695,0.02828137){\color[rgb]{0,0,0}\makebox(0,0)[lt]{\lineheight{1.25}\smash{\begin{tabular}[t]{l}$y_1$\end{tabular}}}}%
    \put(0.52390071,0.02163306){\color[rgb]{0,0,0}\makebox(0,0)[lt]{\lineheight{1.25}\smash{\begin{tabular}[t]{l}$x_7'$\end{tabular}}}}%
    \put(0.84833932,0.00966607){\color[rgb]{0,0,0}\makebox(0,0)[lt]{\lineheight{1.25}\smash{\begin{tabular}[t]{l}$y_3''$\end{tabular}}}}%
  \end{picture}%
\endgroup%
}
\caption{The curves $y_1,x_7',y_3''$ in (\ref{eq:revbreed}), shown on $S_{0,8}$.}
\label{fig:reveng2hurw}
}
\end{figure}

Recall the genus-2 factorization of Baykur-Korkmaz \cite{bk}, which has a lift to $\Mod(\Sigma_2^3)$ (see \cite{huang}, Figure 3). We will first modify this relation by appling the diffeomorphism $\phi$ to each curve (see Figure \ref{fig:BKdiffeo}), giving:
\[ \epsilon z_1z_2z_3bB_2C=d_1d_2d_3.\]
Now Hurwitz move $b$ to the left of $z_3$ and move $C$ to the left of $B_2$ to get
\begin{equation}
\label{eq:newBK}
	\epsilon z_1z_2bz_3'CB_2'=d_1d_2d_3,
\end{equation}
where $z_3'=t_b^{-1}(z_3)$ and $B_2'=t_C^{-1}(B_2)$. These curves are shown in Figure \ref{fig:BK}.

\begin{figure}[h]
\centering{
\resizebox{30mm}{!}{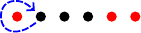}
\caption{A diffeomorphism $\phi$ of $\Sigma_2^3$.}
\label{fig:BKdiffeo}
}
\end{figure}
\begin{figure}[h]
\centering{
\resizebox{110mm}{!}{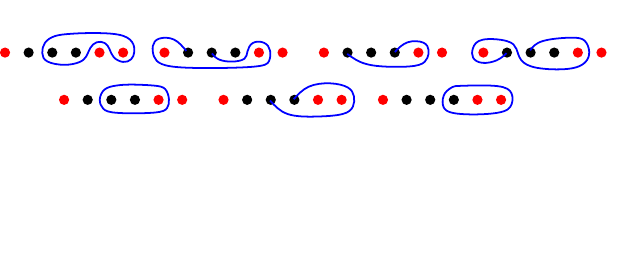}
\caption{The vanishing cycles in (\ref{eq:newBK}).}
\label{fig:BK}
}
\end{figure}

For our next trick, we will start with Xiao's genus-2 fibration of type $(4,3)$ \cite{xiao} and apply partial conjugations along a subword to obtain an infinite family of positive factorizations, each Hurwitz equivalent to Xiao's original fibration. See Figure \ref{fig:Xzero} for the curves appearing in the factorizations.

\begin{prop}
Let $\psi=t_{c_4}$. For each integer $n$, there is a relation
\begin{equation}
\label{eq:new43}
	Q_3'\psi^n(P')\psi^n(R_3)\psi^n(R_2)Q_2R_1Q_1=d_1d_2d_3
\end{equation}
in $\Mod(\Sigma_2^3)$. These factorizations are all Hurwitz equivalent to each other, and in particular the $n=0$ factorization is Hurwitz equivalent to Xiao's $(4,3)$ factorization.
\end{prop}
\begin{rmk}
By \cite{huang}, these factorizations are also equivalent to the $(4,3)$ factorizations of Baykur-Korkmaz and of Hamada.
\end{rmk}
\begin{proof}
Let us first show the $n=0$ case. Starting with Xiao's genus-2 fibration (see \cite{huang}, Figure 7), rotate the entire diagram 180 degrees and apply the diffeomorphism $t_{c_3}$ to each curve to obtain
\[ PR_3Q_3R_2Q_2R_1Q_1=d_1d_2d_3.\]
Now move $R_3$ to the right of $Q_3$ to get
\[ PQ_3'R_3R_2Q_2R_1Q_1=d_1d_2d_3,\]
with $Q_3'=t_{R_3}(Q_3)$. Now move $Q_3'$ to the left of $P$ to get
\[ Q_3'P'R_3R_2Q_2R_1Q_1=d_1d_2d_3,\]
with $P'=t_{Q_3'}^{-1}(P)$. This concludes the $n=0$ case.

\begin{figure}[h]
\centering{
\resizebox{40mm}{!}{
\begingroup%
  \makeatletter%
  \providecommand\color[2][]{%
    \errmessage{(Inkscape) Color is used for the text in Inkscape, but the package 'color.sty' is not loaded}%
    \renewcommand\color[2][]{}%
  }%
  \providecommand\transparent[1]{%
    \errmessage{(Inkscape) Transparency is used (non-zero) for the text in Inkscape, but the package 'transparent.sty' is not loaded}%
    \renewcommand\transparent[1]{}%
  }%
  \providecommand\rotatebox[2]{#2}%
  \newcommand*\fsize{\dimexpr\f@size pt\relax}%
  \newcommand*\lineheight[1]{\fontsize{\fsize}{#1\fsize}\selectfont}%
  \ifx\svgwidth\undefined%
    \setlength{\unitlength}{82.42117454bp}%
    \ifx\svgscale\undefined%
      \relax%
    \else%
      \setlength{\unitlength}{\unitlength * \real{\svgscale}}%
    \fi%
  \else%
    \setlength{\unitlength}{\svgwidth}%
  \fi%
  \global\let\svgwidth\undefined%
  \global\let\svgscale\undefined%
  \makeatother%
  \begin{picture}(1,0.36094442)%
    \lineheight{1}%
    \setlength\tabcolsep{0pt}%
    \put(0,0){\includegraphics[width=\unitlength,page=1]{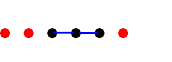}}%
    \put(0.30493229,0.22507327){\color[rgb]{0,0,0}\makebox(0,0)[lt]{\lineheight{1.25}\smash{\begin{tabular}[t]{l}$c_3$\end{tabular}}}}%
    \put(0.48617455,0.03509999){\color[rgb]{0,0,0}\makebox(0,0)[lt]{\lineheight{1.25}\smash{\begin{tabular}[t]{l}$c_4$\end{tabular}}}}%
  \end{picture}%
\endgroup%
}
\caption{The curves $c_3$ and $c_4$ for the Dehn twists $t_{c_3}$ and $\psi=t_{c_4}$.}
\label{fig:Xdiff}
}
\end{figure}

\begin{figure}[h]
\centering{
\resizebox{110mm}{!}{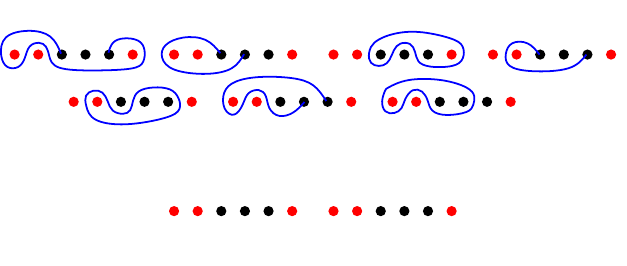}
\caption{The vanishing cycles in the $n=0$ case.}
\label{fig:Xzero}
}
\end{figure}

\begin{figure}[h]
\centering{
\resizebox{130mm}{!}{
\begingroup%
  \makeatletter%
  \providecommand\color[2][]{%
    \errmessage{(Inkscape) Color is used for the text in Inkscape, but the package 'color.sty' is not loaded}%
    \renewcommand\color[2][]{}%
  }%
  \providecommand\transparent[1]{%
    \errmessage{(Inkscape) Transparency is used (non-zero) for the text in Inkscape, but the package 'transparent.sty' is not loaded}%
    \renewcommand\transparent[1]{}%
  }%
  \providecommand\rotatebox[2]{#2}%
  \newcommand*\fsize{\dimexpr\f@size pt\relax}%
  \newcommand*\lineheight[1]{\fontsize{\fsize}{#1\fsize}\selectfont}%
  \ifx\svgwidth\undefined%
    \setlength{\unitlength}{286.11737589bp}%
    \ifx\svgscale\undefined%
      \relax%
    \else%
      \setlength{\unitlength}{\unitlength * \real{\svgscale}}%
    \fi%
  \else%
    \setlength{\unitlength}{\svgwidth}%
  \fi%
  \global\let\svgwidth\undefined%
  \global\let\svgscale\undefined%
  \makeatother%
  \begin{picture}(1,0.1251142)%
    \lineheight{1}%
    \setlength\tabcolsep{0pt}%
    \put(0,0){\includegraphics[width=\unitlength,page=1]{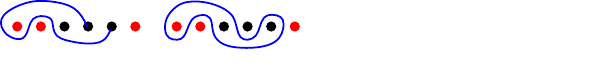}}%
    \put(0.02672512,0.01132963){\color[rgb]{0,0,0}\makebox(0,0)[lt]{\lineheight{1.25}\smash{\begin{tabular}[t]{l}$t_{Q_1}^{-1}(R_1)$\end{tabular}}}}%
    \put(0.29422183,0.01132963){\color[rgb]{0,0,0}\makebox(0,0)[lt]{\lineheight{1.25}\smash{\begin{tabular}[t]{l}$t_{\psi(R_1)}^{-1}(Q_1)$\end{tabular}}}}%
  \end{picture}%
\endgroup%
}
\caption{}
\label{fig:X2}
}
\end{figure}

Recall that we set $\psi=t_{c_4}$. Observe that $t_{Q_1}^{-1}(R_1)=\psi(R_1)$ and $t_{\psi(R_1)}^{-1}(Q_1)=\psi(Q_1)$, which can be verified by direct computation (see Figure \ref{fig:X2}). It follows that for each integer $n$,
\[ \psi^{n+1}(R_1)=\psi^nt_{Q_1}^{-1}(R_1)=\psi^nt_{Q_1}^{-1}\psi^{-n}(\psi^n(R_1))=t_{\psi^n(Q_1)}^{-1}(\psi^n(R_1))\]
and that
\[ \psi^{n+1}(Q_1)=\psi^nt_{\psi(R_1)}^{-1}(Q_1)=\psi^nt_{\psi(R_1)}^{-1}\psi^{-n}(\psi^n(Q_1))=t_{\psi^{n+1}(R_1)}^{-1}(\psi^n(Q_1)).\]
Using this, we have the following sequence of Hurwitz moves:
\[ t_{\psi^n(R_1)}t_{\psi^n(Q_1)}=t_{\psi^n(Q_1)}t_{\psi^{n+1}(R_1)}=t_{\psi^{n+1}(R_1)}t_{\psi^{n+1}(Q_1)},\]
so that $t_{\psi^n(R_1)}t_{\psi^n(Q_1)}$ and $t_{\psi^{n+1}(R_1)}t_{\psi^{n+1}(Q_1)}$ are Hurwitz equivalent for each $n$. Hence $t_{\psi^n(R_1)}t_{\psi^n(Q_1)}$ is Hurwitz equivalent to $t_{R_1}t_{Q_1}$ for each $n$.

We finish the proof by taking the $n=0$ case of (\ref{eq:new43}) and doing a global conjugation by $\psi^n$:
\[ Q_3'\psi^n(P')\psi^n(R_3)\psi^n(R_2)Q_2\psi^n(R_1)\psi^n(Q_1)=d_1d_2d_3.\]
$Q_3'$ and $Q_2$ are disjoint from $c_4$ and are hence unchanged by $\psi^n$. Now the term $\psi^n(R_1)\psi^n(Q_1)$ is Hurwitz equivalent to $R_1Q_1$, which proves (\ref{eq:new43}) as desired.
\end{proof}

We use the curves in Figure \ref{fig:genus3} to breed the relations (\ref{eq:newBK}) and (\ref{eq:new43}). Embed $\Sigma_2^3$ into $\Sigma_3^4$ by sending $d_1\mapsto C$, $d_2\mapsto\delta_4$, and $d_3\mapsto\delta_2$. Under this embedding, the relation (\ref{eq:newBK}) becomes $ \sigma_1'y_1x_2''\beta y_3''C'x_7'=C\delta_4\delta_2$. Doing a cyclic permutation and moving $C'$ to the right gives
\[ x_7'\sigma_1'y_1x_2''\beta y_3''=C\delta_4\delta_2C'^{-1}.\]
Similarly, we can embed $\Sigma_2^3$ into $\Sigma_3^4$ by sending $d_1\mapsto\delta_1$, $d_2\mapsto\delta_3$, and $d_3\mapsto C'$. Under this embedding, the $n=1$ case of the relation (\ref{eq:new43}) becomes $Cx_4'x_6'x_8dx_5e=\delta_1\delta_3C'$. Moving $C$ to the right gives
\[ x_4'x_6'x_8dx_5e=\delta_1\delta_3C'C^{-1}.\]
Breeding these two relations recovers our relation (\ref{eq:revbreed}) that we obtained by reverse engineering.

\end{document}